\documentclass[11pt,a4paper,reqno]{amsart}
\usepackage{amsmath,amssymb,fullpage,tabls,graphicx,color}
\usepackage{pdfsync}
\usepackage{ifpdf}
\ifpdf
\fi
\usepackage[T1]{fontenc}
\usepackage{textcomp}

\newtheorem{teo}{Theorem}[section]
\newtheorem{prop}[teo]{Proposition}

\newtheorem{lemma}[teo]{Lemma}

\newtheorem{example}[teo]{Example}

\newtheorem{oss}[teo]{Remark}

\newcommand{\compact}{\subset\subset}    % in amssymb package
\newcommand{\de}{\partial}
\newcommand{\om}{\omega}
\newcommand{\Om}{\Omega}

\newcommand{\eps}{\varepsilon}
\newcommand{\e}{\varepsilon}

\newcommand{\g}{\gamma}
\newcommand{\gp}{\dot{\gamma}}
\newcommand{\gpp}{\ddot{\gamma}}

\numberwithin{equation}{section}

\def\step#1#2{\par\noindent{\underline{\it Step~#1.}}\emph{ #2}\\}
\newcommand{\strip}{{\mathcal S}}
\renewcommand{\bowtie}{\mathcal W}
\newcommand{\pino}{\mathcal P}
\newcommand{\ears}{\mathcal Q}

\newcommand{\mdiv}{\mathop{\mathrm{div}}}

\newcommand{\difsim}{\Delta\,}

\DeclareMathOperator*{\dist}{dist}

\DeclareMathOperator*{\reach}{\mathcal R}

\newcommand{\R}{{\mathbb R}}

\newcommand{\N}{{\mathbb N}}

\newcommand{\Hau}{{\mathcal H}}
\def\H{{\mathcal H}}

\title{On the Cheeger sets in strips and non-convex domains}

\author[G.P.~Leonardi]{Gian Paolo Leonardi}
\address{Dipartimento di Scienze Fisiche, Informatiche e Matematiche, Universit{\`a} di Modena e Reggio Emilia, Via Campi 213/b, 41100
    Modena, Italy}
\email{gianpaolo.leonardi@unimore.it}

\author[A.~Pratelli]{Aldo Pratelli}
\address{Department of Mathematics, University of Erlangen, Cauerstrasse 11, 90158 Erlangen, Germany}
\email{pratelli@math.fau.de}


\begin{document}
\begin{abstract}
In this paper we consider the Cheeger problem for non-convex domains, with a particular interest in the case of planar strips, which have been extensively studied in recent years. Our main results are an estimate on the Cheeger constant of strips, which is stronger than the previous one known from~\cite{KrePra2011}, and the proof that strips share with convex domains a number of crucial properties with respect to the Cheeger problem. Moreover, we present several counterexamples showing that the same properties are not valid for generic non-convex domains.
\end{abstract}

\maketitle

\section{Introduction}

The Cheeger problem is an isoperimetric-like problem, which has been extensively studied in the last decades in different contexts, ranging from Riemannian Geometry to Calculus of Variations, also because of its connections with eigenvalue problems. Its formulation is very simple: given a domain $\Omega\subseteq \R^n$, with $n\geq 2$, one is asked to compute or estimate its \emph{Cheeger constant}, defined as
\[
h(\Om) := \inf \left\{\frac{P(F)}{|F|}\,:\ F\subseteq\Om,\ |F|>0\right\}\,,
\]
where $|F|$ and $P(F)$ denote respectively the volume and perimeter of a Borel set $F\subseteq\R^n$. Despite this so simple formulation, many non-trivial questions arise, and they are object of deep investigation. For quite a large class of domains $\Omega$, it is possible to say that the above infimum is actually a minimum, and in this case each set $E$ realizing the minimum is called a \emph{Cheeger set in $\Omega$}; if the whole set $\Omega$ is a minimizer, then it is simply called a \emph{Cheeger set}; it is clear that any set $E$ which is a Cheeger set for some domain $\Omega\supseteq E$, is also a Cheeger set. One usually refers to the ``Cheeger problem'' both for the computation, or the estimation, of the constant $h(\Omega)$, and for the characterization of Cheeger sets in $\Omega$.\par

Among various geometric properties, that will be discussed in more detail later, a particularly useful one is that the free boundary of any Cheeger set $E$ in $\Omega$, that is, the part of $\partial E$ which is in the interior of $\Omega$, is the union of a singular set of Hausdorff dimension at most $n-8$ and a smooth hypersurface with constant mean curvature, and this curvature coincides with the inverse of the constant $h(\Omega)$. This is particularly useful in the two-dimensional case, because then one deduces that $\partial E \cap \Omega$ is made by arcs of circle of radius $1/h(\Omega)$, which is a very strong geometric constraint. Another very important case is the one of convex domains. In particular, for a convex $2$-dimensional domain $\Omega$, it is known that a Cheeger set $E$ exists, is unique, and it coincides with the union of all balls of radius $1/h(\Omega)$ which are contained in $\Omega$, thus $E=E_r+B_r$ where $r=1/h(\Omega)$ and $E_r$ is the set of all points of $\Omega$ having distance at least $r$ from $\partial\Omega$; moreover, the \emph{inner Cheeger formula} $|E_r|=\pi r^2$ for the area of $E_r$ holds. All these properties are extremely useful, in particular the ``union of balls property'' completely characterizes the Cheeger set once the Cheeger constant is known, and on the other hand the inner Cheeger formula allows to explicitly compute $r$, thus $h(\Omega)$, because when $r$ increases the set $E_r$ become smaller and smaller, thus there is a unique $r$ such that the inner Cheeger formula holds.\par

The interested reader can look for instance~\cite{ButCarCom2007,CarComPey2009,CasChaMolNov2008,CasFacMei2009,CasMirNov2010,IonLac2005,Strang2010}, where further applications, developments and extensions of the Cheeger problem are considered (see also the survey paper \cite{Leonardi2014}).\par\medskip

In this paper we will mainly deal with the case of strips, which are basically bended rectangles, that is, curved two-dimensional tubes. The strips, or their three-dimensional counterpart, namely, the waveguides, are quite studied since several years, in particular they are crucial in applications, for instance in engineering and in medicine for their optical properties. Understanding the spectral properties of waveguides or strips is very important, and --as said above-- this is deeply connected with the Cheeger problem. Particularly relevant is the limit case when strips or waveguides become extremely thin; up to a rescaling, it is equivalent but notationally simpler to assume that the thickness is fixed and the length explodes. Some more information on waveguides and their importance can be found in~\cite{DucExn1995,KreKri2005} and in the references therein, while the mathematical study of the Cheeger problem for strips has been also carried on in~\cite{KrePra2011}, see also~\cite{PrSa14}. The main goal of this paper is to observe that strips behave like convex domains, even if they are not convex; in other words, this shows that bending a rectangle to transform it into a strip does not change too much the situation for what concerns the Cheeger problem. More precisely, we will show that all the properties listed above and valid for convex sets, are actually valid also for strips; as a consequence, we derive an estimate from above and below for the Cheeger constant of a strip, which is stronger than the analogous estimate found in~\cite{KrePra2011}. We will present also several examples, some of which new, to underline that the above-mentioned properties are not valid for generic sets, but they are really specific for convex sets and for strips.\par\medskip

The plan of the paper is the following. In Section~\ref{section:basic} we briefly recall the definition of perimeter and its main properties, then we set the Cheeger problem and we list some of the main results: these are well-known, with the exception of Proposition \ref{prop:CheegerGenProp} (vii) (a weak form of a well-known ``tangential contact property'' of Cheeger sets) and of Theorem \ref{teo:hcontinua} (two continuity properties of the Cheeger constant with respect to $L^{1}$ and BV-strict topologies), and mostly focused on the convex two-dimensional case. Then we show a technical but useful result about ``rolling balls'', Lemma~\ref{lemma:movingball}. Section~\ref{section:strip} is the main section of the paper: here we give the definition of the strips and we prove all results concerning them. Finally, in Section~\ref{sect:examples} we collect several examples, some well-known and other original (and possibly interesting in themselves), in order to stress the peculiarity of convex domains, as well as of strips, with respect to the Cheeger problem.

\section{Some preliminary results on the Cheeger problem\label{section:basic}}

This Section is devoted to present some important facts about the Cheeger problem. After a first, very brief introduction on the concept of perimeter, we will start, in Section~\ref{sectdef}, with the relevant definitions and the most important, general properties of the Cheeger problem (in addition to the well-known ones, we prove two facts, namely Proposition \ref{prop:CheegerGenProp} (vii) and Theorem \ref{teo:hcontinua}, that we have not been able to locate in the existing literature).  Then, in Section~\ref{sectconv}, we will discuss some known facts concerning the convex case, and finally, in Section~\ref{sectnew}, we will give a couple of technical results valid in the planar case: even though they seem quite intuitive, we could not find them written anywhere.

\subsection{Basic facts about the perimeter}

Let us start with a very quick introductory section about the notion of perimeter. The initiated reader can clearly skip it, but also any other reader can do the same, and limit himself to read this paper considering the case of sufficiently regular sets, where the notion of perimeter is well-known.\par

Given any Borel set $E\subseteq\R^n$, we will denote by $|E|$ its Lebesgue measure (which we will also shortly call \emph{volume}), and by $P(E)$ its \emph{perimeter}. While for a sufficiently smooth set, for instance a set with Lipschitz boundary, or a polyhedron, the standard concept of perimeter is simply the $(n-1)$-area of the boundary, in the general framework of Borel sets this doesn't work properly. The correct, general definition of perimeter of a Borel set $E\subseteq \R^n$ is then the following,
\[
P(E) := \sup\left\{\int_{E}\mdiv g:\ g\in C^{1}_{c}(\R^{n};\R^{n}),\ |g|\leq 1\right\}\,.
\]
Whenever $P(E)<+\infty$, we will say that $E$ is a \emph{set of finite perimeter}; if this is the case, it can be proved that the perimeter of $E$ coincides with the total variation of the distributional gradient of the characteristic function of $E$, namely,
\[
P(E) = |D\chi_{E}|(\R^{n})\,.
\]
This allows us to define also the \emph{relative perimeter} $P(E;A) := |D\chi_{E}|(A)$ for any pair of Borel sets $A,\,E\subseteq \R^{n}$; roughly speaking, the relative perimeter $P(E;A)$ measures how big is the boundary of $E$ inside $A$. An immediate application of the Radon-Nikodym Theorem yields the existence of a Borel vector-valued function $\nu_{E}$, the \emph{outer normal} of $E$, such that $|\nu_{E}| = 1$ $|D\chi_{E}|$-almost everywhere, and
\[
D\chi_{E} = -\nu_{E} |D\chi_{E}|\,.
\]
An important, equivalent definition of $\nu_E$ is the following: we define the \emph{reduced boundary} $\de^{*}E$ as the set of all the points $x\in \R^n$ such that $0<|E\cap B_r(x)| < \om_{n}r^{n}$ for all $r>0$ and the limit
\[
\lim_{r\to 0^{+}}\frac{D\chi_{E}(B_{r}(x))}{|D\chi_{E}|(B_{r}(x))}
\]
exists and has norm $1$ (here, and in the following, we will always denote by $B_r(x)$ the ball centered at $x$ with radius $r$, and set $\om_{n} = |B_{1}(0)|$). One can prove that this limit exists $|D\chi_{E}|$-almost everywhere and coincides with $-\nu_E(x)$ for $|D\chi_{E}|$-almost each point of $\partial^* E$. It is immediate to observe that, whenever $E$ is smooth enough, then the reduced boundary is nothing else than the usual topological boundary, and the outer normal coincides with the exterior normal vector to the boundary. However, for a general Borel set the next classical result by De Giorgi holds~\cite{DeGiorgi2006}.
\begin{teo}[De Giorgi]\label{teo:degiorgi}
Let $E$ be a set of finite perimeter, then 
\begin{itemize}
\item[(i)] $\de^{*}E$ is countably $\Hau^{n-1}$-rectifiable in the sense of Federer~\cite{FedererBOOK};

\item[(ii)] for any $x\in \de^{*}E$, $\chi_{t(E-x)} \to \chi_{H_{\nu_{E}(x)}}$ in $L^{1}_{loc}(\R^{n})$ as $t\to +\infty$, where $H_{\nu}$ denotes the half-space through $0$ whose exterior normal is $\nu$;

\item[(iii)] for any Borel set $A$, $P(E;A) = \Hau^{n-1}(A\cap \de^{*}E)$, thus in particular $P(E)=\H^{n-1}(\partial^* E)$;

\item[(iv)] $\int_{E}\mdiv g = \int_{\de^{*}E} g\cdot \nu_{E}\, d\Hau^{n-1}$ for any $g\in C^{1}_{c}(\R^{n};\R^{n})$.
\end{itemize}
\end{teo}

As we noticed above, these notions extend perfectly the classical notions whenever a set is smooth enough. In particular, for a regular set one has $P(E)=\H^{n-1}(\partial E)=\H^{n-1}(\partial^* E)$; for a general Borel set it is always true, by the Theorem above, that $P(E)=\H^{n-1}(\partial^* E)$, but this does not need to coincide with the $\H^{n-1}$ measure of the topological boundary.

There are many advantages of using this generalized notion of perimeter. First of all, this applies to any Borel set; moreover, whenever two sets coincide almost everywhere, then their reduced boundaries --and so, their perimeters-- coincide (while the topological boundaries can be completely different). Another important feature is the validity of the following well-known lower-semicontinuity and compactness properties (see, e.g., \cite{AmbFusPal2000}):
\begin{prop}[Lower-semicontinuity and compactness]\label{prop:semicomp}
Let $\Om\subseteq\R^{n}$ be an open set and let $(E_{j})_{j}$ be a sequence of Borel sets. We have the following well-known properties:
\begin{itemize}
\item[(i)] if $E$ is a Borel set, such that $\chi_{E_{j}}\to\chi_{E}$ in $L^{1}_{loc}(\Om)$, then $P(E;\Om) \leq \liminf\limits_{j} P(E_{j};\Om)$;
\item[(ii)] if there exists a constant $C>0$ such that $P(E_{j};\Om)\leq C$ for all $j$, then there exists a subsequence $E_{j_{k}}$ and a Borel set $E$ such that $\chi_{E_{j_{k}}}\to \chi_{E}$ in $L^{1}_{loc}(\Om)$.
\end{itemize}
\end{prop}
Other useful properties of the perimeter (invariance by isometries and scaling property, isoperimetric inequality, lattice property) are collected in the next proposition. 
%There, and in the following, by $\omega_n=|B_1(0)|$ we denote the volume of the unit ball.
\begin{prop}
Given two Borel sets $E,F\subseteq\R^{n}$ of finite perimeter, $\lambda>0$ and an isometry $T:\R^{n}\to \R^{n}$, we have
\begin{gather}
\label{Pomogeneo}
P(\lambda T(E)) = \lambda^{n-1}P(E)\,,\\[3pt] 
\label{isopRn}
P(E) \geq n\om_{n}^{1/n}|E|^{\frac{n-1}{n}}\,,\\[3pt]
\label{reticolo}
P(E\cup F) + P(E\cap F) \leq P(E) + P(F)\,.
\end{gather}
\end{prop}

While generic sets of finite perimeter can be also very weird, a regularity theory is available in particular for \emph{minimizers} of the perimeter subject to a volume constraint (see~\cite{Tamanini1982}). 
\begin{teo}[Regularity of perimeter minimizers with volume constraint]\label{teo:regolarita}
Let $\Om$ be a fixed open domain, and assume that $E$ is a Borel set satisfying the following property: $P(E;\Om)<+\infty$ and for all Borel $F$ such that $E\difsim F \compact\Om$ and $|F\cap \Om| = |E\cap \Om|$, it holds
\[
P(F;\Om) \leq P(E;\Om)\,.
\]
Then, $\de^{*}E \cap \Om$ is an analytic surface with constant mean curvature, and the singular set $(\de E \setminus \de^{*}E) \cap \Om$ is a closed set with Hausdorff dimension at most $n-8$. 
\end{teo}

\subsection{Setting of the Cheeger problem and first main results\label{sectdef}}

In this section we introduce the Cheeger problem and we list some well-known results. The definition of the problem is the following: for any bounded, open set $\Omega$, we define the \emph{Cheeger constant of $\Om$} as
\begin{equation}\label{Cheegerconst}
h(\Om) := \inf \left\{\frac{P(F)}{|F|}\,:\ F\subseteq\Om,\ |F|>0\right\}\,.
\end{equation}
Any set $F\subseteq \Omega$ which realizes the above infimum is called a \emph{Cheeger set in $\Omega$}, and a set is simply called a \emph{Cheeger set} when it realizes the above infimum itself. Observe that of course if $F$ is a Cheeger set in $\Omega$, then it is also a Cheeger set. The following are some of the basic properties which are known about the problem, the proof of most of them can be found for instance in~\cite{StrZie1997, KawFri2003,KawLac2006} as we discuss below.

\begin{prop}\label{prop:CheegerGenProp}
Let $\Om,\widetilde \Om\subseteq\R^{n}$ be bounded, open sets. Then the following properties hold.
\begin{itemize}
\item[(i)] If $\Om\subseteq\widetilde\Om$ then $h(\Om) \geq h(\widetilde\Om)$.

\item[(ii)] For any $\lambda>0$ and any isometry $T:\R^{n}\to\R^{n}$, one has $h(\lambda T(\Om)) = \frac{1}{\lambda} h(\Om)$.

\item[(iii)] There exists a (possibly non-unique) Cheeger set $E\subseteq\Om$.

\item[(iv)] If $E$ is Cheeger in $\Om$, then it minimizes the relative perimeter among subsets of $\Omega$ with the same volume as $E$; consequently, $\de^* E\cap \Om$ has the regularity stated in Theorem~\ref{teo:regolarita}, and in particular $\de^{*}E\cap \Om$ is a hypersurface of constant mean curvature equal to $\frac{h(\Om)}{n-1}$.

\item[(v)] If $E$ is Cheeger in $\Om$ then $|E| \geq \om_{n}\left(\frac{n}{h(\Om)}\right)^{n}$.

\item[(vi)] If $E$ and $F$ are Cheeger in $\Om$, then $E\cup F$ and $E\cap F$ (if it is not empty) are also Cheeger in $\Om$.

\item[(vii)] If $E$ is Cheeger in $\Om$ and $\Om$ has finite perimeter, then $\de E\cap \Om$ can meet $\de^{*}\Om$ only in a tangential way, that is, for any $x\in \de^{*}\Om\cap \de E$ one has that $x\in \de^{*}E$ and $\nu_{E}(x) = \nu_{\Om}(x)$.

\end{itemize}
\end{prop}
The claims~(i)--(vi) are simple and widely known, and can be found in the previously cited references. More precisely, (i) and (ii) are immediate consequences of the definition of Cheeger constant and of~\eqref{Pomogeneo} coupled with $|\lambda \Om| = \lambda^{n}|\Om|$. The proofs of (iii) and (iv) are accomplished by, respectively, Proposition~\ref{prop:semicomp} and Theorem~\ref{teo:regolarita}. The proof of~(v) follows from the isoperimetric inequality~\eqref{isopRn} and the fact that $P(E) = h(\Om)|E|$. Finally, to prove (vi) it is enough to apply~\eqref{reticolo} and get
\[\begin{split}
h(\Om) (|E\cup F| + |E\cap F|) &= h(\Om)(|E|+|F|)
= P(E)+P(F)\\
&\geq P(E\cup F) + P(E\cap F)
\geq h(\Om)(|E\cup F| + |E\cap F|)\,,
\end{split}\]
hence all previous inequalities are actually equalities and this readily proves (vi). Instead, concerning the property~(vii), it is widely known if the boundary of $\Omega$ is smooth in a neighborhood of $x$ (see for instance~\cite{StrZie1997}) but, as far as we know, it has not yet been written in the fully general form as above. For the sake of completeness, we add in the Appendix a proof of it.

\begin{oss}\rm 
We notice that, by Proposition~\ref{prop:CheegerGenProp} (iii), (v) and (vi), we can always find minimal Cheeger sets in $\Om$ (possibly not unique) and a unique maximal Cheeger set (this last can be obtained as the union of all minimal Cheeger sets of $\Om$). An example of a domain with two disjoint minimal Cheeger sets is shown in Figure~\ref{fig:duecheeger}.
\end{oss}

We consider now the problem of continuity of the Cheeger constant $h(\Om)$ with respect to some suitable notions of convergence of domains. In Theorem~\ref{teo:hcontinua} below, we show that the Cheeger constant is lower semicontinuous with respect to $L^{1}$-convergence of domains, while it is continuous if we additionally assume that the perimeters of the approximating domains converge to the perimeter of the limit domain. We point out that the continuity of the Cheeger constant under $L^{1}$-convergence of \emph{convex} domains has been proved in~\cite{Parini_doktorarbeit2009} (since $L^{1}$-convergence plus convexity implies convergence of the perimeters, the continuity result for convex domains is a particular case of Theorem~\ref{teo:hcontinua}).
\begin{teo}[Continuity of the Cheeger constant]\label{teo:hcontinua}
Let $\Om,\Om_{j}\subseteq\R^{n}$ be nonempty open bounded sets for all $j\in \N$. If $\chi_{\Om_{j}} \to \chi_{\Om}$ in $L^{1}$, then 
\begin{equation}\label{hlsc}
\liminf_{j\to\infty} h(\Om_{j}) \geq h(\Om)\,.
\end{equation}
If in addition $\Om, \Om_{j}$ are sets of finite perimeter and $P(\Om_{j})\to P(\Om)$ as $j\to\infty$, then
\begin{equation}\label{hc}
\lim_{j\to\infty} h(\Om_{j}) = h(\Om)\,.
\end{equation}
\end{teo}
\begin{proof}
Let $E_{j}$ be a Cheeger set in $\Om_{j}$ (whose existence is guaranteed by Proposition~\ref{prop:CheegerGenProp} (iii)). Without loss of generality we assume that $\liminf\limits_{j\to\infty} P(E_{j})$ is finite, then by Proposition~\ref{prop:semicomp} we deduce that $\chi_{E_{j}}\to \chi_{E}$ in $L^{1}$ as $j\to\infty$, up to subsequences and for some Borel set $E$ with positive volume. Since $E_{j}\subseteq\Om_{j}$ and $\chi_{\Om_{j}}\to \chi_{\Om}$ in $L^{1}$ as $j\to\infty$, one immediately infers that $E\subseteq\Om$ up to null sets. Then by Proposition~\ref{prop:semicomp} and by the convergence of $|E_{j}|$ to $|E|$, one has
\[
h(\Om) \leq \frac{P(E)}{|E|} \leq \liminf_{j\to\infty}\frac{P(E_{j})}{|E_{j}|}=\liminf_{j\to\infty} h(\Om_j)\,,
\] 
which is~\eqref{hlsc}. If in addition $P(\Om_{j})\to P(\Om)$ as $j\to\infty$, then we consider $E$ Cheeger in $\Om$ and define $E_{j} = \Om_{j}\cap E$. One can easily check that $E_{j} \to E$ and $E\cup \Om_{j}\to \Om$ in $L^{1}$, as $j\to\infty$. Therefore by~\eqref{reticolo} we find
\[
\limsup_{j\to\infty} P(E_{j}) \leq P(E) + \limsup_{j\to\infty} P(\Om_{j}) - \liminf_{j\to\infty} P(E\cup \Om_{j})
\leq P(E) + P(\Om) - P(\Om)
= P(E)\,,
\]
which combined with~\eqref{hlsc} gives~\eqref{hc}.
\end{proof}

\subsection{The Cheeger problem in convex domains\label{sectconv}} 
In the particular case when the domain $\Om$ is convex, several further properties are known; in this section, we list some of the most interesting ones. Some of these properties will be later generalized to the case of the strips, which we will introduce in Section~\ref{section:strip}.\par

The first property, which can be found in~\cite{AltCas2009} (see also the references therein), is the following uniqueness and convexity result.
\begin{teo}\label{teo:convC11unique}
Let $\Om\subseteq\R^{n}$ be a convex domain. Then there exists a unique Cheeger set $E$ in $\Om$. Moreover, $E$ is convex and of class ${\rm C}^{1,1}$.
\end{teo}

While the previous result holds for any dimension $n\geq 2$, a quite more precise statement holds in the planar case, see~\cite{StrZie1997, KawLac2006}. We will use the following notation: for any set $\Omega$, and any $r>0$, we write
\[
E_{r}:= \{x\in \Om:\ \dist(x,\de \Om)>r\}\,;
\]
in particular, if $r=h(\Omega)^{-1}$, then the set $E_r$ is referred to as the \emph{inner Cheeger set} of $\Omega$, the reason being clear from the next result.
\begin{teo}\label{teo:kawlac}
Let $\Om$ be a bounded convex set in $\R^{2}$. Then the unique Cheeger set $E$ of $\Om$ is the union of all balls of radius $r = h(\Om)^{-1}$ that are contained in $\Om$, hence $E = E_{r}+B_r(0)$. Moreover, it holds
\[
|E_{r}| = \pi r^{2}\,.
\]
\end{teo}
The proof of Theorem~\ref{teo:kawlac} is essentially based on \emph{Steiner's formulae} for area and perimeter of tubular neighbourhoods of convex sets in the plane (\cite{Steiner1840}): if $A\subseteq\R^{2}$ is a bounded convex set and $\rho>0$, then setting $A^{\rho} = A + B_{\rho}(0)$ we have
\begin{gather}
\label{Asteiner}
|A^{\rho}| = |A| + \rho\, P(A) + \pi \rho^{2},\\ 
\label{Psteiner}
P(A^{\rho}) = P(A) + 2\pi \rho\,.
\end{gather}
Some generalizations of Steiner's formulae have been proved, for instance by Weyl in the case of $n$-dimensional domains with $C^{2}$ boundary (the so-called tube formula, see~\cite{Weyl1939}) and then by Federer~\cite{Federer1959} under the assumption of \emph{positive reach}. Let us be more precise: given $K\subseteq\R^{n}$ compact, we define the \emph{reach} of $K$ as
\[
 \reach(K) = \sup\{\e\ge 0:\ \text{if $\dist(x,K)\le \e$ then $x$ has a unique projection onto $K$}\}\,.
\] 
We say that $K$ has positive reach if $\reach(K)>0$. Notice that if $K$ is convex, then $\reach(K) = +\infty$. It is convenient to introduce the \emph{outer Minkowski content} of an open bounded set $A$, defined as
\[
\mathcal M(A) = \lim_{\rho\to 0}\frac{|A^{\rho}|-|A|}{\rho}\,,
\]
provided that the limit exists. Then the following result holds.
\begin{prop}\label{steinerformulas}
Let $A\subseteq\R^{2}$ be a bounded open set with Lipschitz boundary and assume that $\overline{A}$ has positive reach. Then Steiner's formulae~\eqref{Asteiner} and~\eqref{Psteiner} hold for all $0<\rho\leq \reach(\overline{A})$.  
\end{prop}
\begin{proof}
Let $d_{A}$ denote the distance function from $A$. The positive reach assumption implies that $d_{A}$ is of class ${\rm C}^{1,1}_{loc}$ on the open set $\{x:\ 0<d_{A}(x)<\reach(\overline{A})\}$. Moreover its Jacobian is $|\nabla d_{A}(x)|=1$ on this set. Since Weil's extension of the Steiner formula~(\ref{Asteiner}) in particular holds true for ${\rm C}^{1,1}$ planar domains, for any $0<t<\rho<\reach(A)$ we can use coarea formula getting
\[
|A^{\rho}| - |A^{t}| = \int_{t}^{\rho}P(A^s)\, ds = \int_{t}^{\rho }\big(P(A^{t})+2\pi(s-t)\big)\, ds
= (\rho-t)P(A^{t}) + \pi(\rho-t)^{2}\,.
\]
Now, since $|A^{t}| \to |A|$ as $t\to 0$, we obtain the existence of $\lim\limits_{t\to 0} P(A^{t})$ and, calling $P_{0}$ this limit, it holds
\[
\frac{|A^{\rho}|-|A|}{\rho} - \pi \rho = P_{0}\,,\qquad \forall\,0<\rho<\reach(A)\,.
\]
Letting now $\rho\to 0$, we deduce that $\mathcal M(A) = P_{0}$. On the other hand, being $\overline{A}$ of positive reach and with Lipschitz boundary, we can apply Theorem 9 in~\cite{AmbColVil2008} to deduce that $\mathcal M(A) = P(A)$. Then for all $0<\rho<\reach(\overline{A})$ we obtain
\[
P(A) = \mathcal M(A) = P_0 = \frac{|A^{\rho}|-|A|}{\rho} - \pi \rho\,.
\]
By multiplying this last equality by $\rho$ we obtain~\eqref{Asteiner}, while by differentiating~\eqref{Asteiner} we get~\eqref{Psteiner}. The validity also for $\rho=\reach(\overline A)$ follows then by continuity.
\end{proof}

\subsection{Some further results about Cheeger sets in $\R^{2}$\label{sectnew}}
Let $E$ be a Cheeger set inside an open bounded domain $\Om\subseteq\R^{2}$, and set $r = h(\Om)^{-1}$ as before. Then a first, general fact is that any connected component of $\de E\cap \Om$ is an arc of radius $r$, that cannot be longer than $\pi r$ (i.e., it can be at most a half-circle).
\begin{lemma}\label{lemma:180}
Any connected component  $S$ of $\de E\cap \Om$ is an arc of circle of radius $r$, whose length does not exceed $\pi r$.
\end{lemma}
\begin{proof}
By Proposition~\ref{prop:CheegerGenProp}~(iv) we already know that $S$ is a circular arc of radius $r$. Assume then by contradiction that this arc has length $l>\pi r$. Consider the arc $S' \subseteq S$ with length equal to $\pi r$ and whose mid-point coincides with that of $S$: notice that $S'$ lies entirely in the interior of $\Om$. Call now $x$ the center of the corresponding circle, and let now $S''$ be the arc of circle centered at $x_\eps:= x+r\tan(\eps)\nu$, connecting the two endpoints of $S'$ and having opening angle equal to $\pi+2\eps$. Since $S'$ is contained in the open set $\Om$, the same holds true for $S''$ if $\eps$ is small enough. For such a small $\eps$, let us then define the competitor $E_\eps$, slightly bigger than $E$, whose boundary coincides with $(\partial E \setminus S') \cup S''$. Now we compute the perimeter and the area of $E_{\e}$: we have
\[
P(E_{\e}) = P(E) -\pi r + (\pi+2\eps) \, \frac r{\cos \eps}
= \frac{|E|}r +2r \eps  + \frac{\pi r}2\,  \eps^2+ o(\eps^3)\,,
\]
while
\[
|E_{\e}| = |E| - \frac \pi 2 \, r^2 + \frac{\pi+2\eps}2 \cdot \frac{r^2}{\cos^2\eps} + r^2\tan\eps
=|E| +2r^2 \eps+ \frac{\pi r^2}2\, \eps^2+ \frac{r^2}3\, \eps^3+ o(\eps^3)\,.
\]
Therefore, the Cheeger ratio of $E_{\e}$ satisfies
\[
\frac{P(E_{\e})}{|E_{\e}|} < \frac 1r = h(\Om),
\]
which contradicts the definition of $h(\Om)$.
\end{proof}

A seemingly reasonable property of a planar Cheeger set $E$ is the fact that $E$ satisfies an internal ball condition of radius $r = \frac{|E|}{P(E)}$, or that it is a union of balls of radius $r$ (this second property is slightly stronger). This fact is false in general (see Figure~\ref{fig:cheegerinunion} and, in particular, Example~\ref{bow-tie}). Anyway, the following result holds true: if a maximal Cheeger set $E$ in $\Om$ contains some ball $B_{r}(x_{0})$, then it contains also all the balls which can be obtained by ``rolling'' $B_r(x_0)$ inside $\Om$.
\begin{lemma}[Rolling ball]\label{lemma:movingball}
Assume that the maximal Cheeger set $E$ in $\Om$ contains a ball $B_{r}(x_{0})$, being $r=1/h(\Om)$, and let $\gamma:[0,1]\to \Om$ be a ${\rm C}^{1,1}$ curve, with curvature bounded by $h(\Om)$, such that $\gamma(0)=x_{0}$ and $B_{r}(\gamma(t))\subseteq\Om$ for all $t\in [0,1]$. Then $B_{r}(\gamma(t))\subseteq E$ for all $t\in [0,1]$.
\end{lemma}
\begin{proof}
Let $t^*\in [0,1]$ be the largest time for which $B_r(\gamma(t))\subseteq E$ for all $t\in [0,t^{*}]$, set $B^{*} = B_r(\gamma(t^{*}))$ for brevity, and assume by contradiction that $t^*<1$. Take now $t\in (t^*,1)$, very close to $t^*$ and such that $B_r(\gamma(t))$ is not contained in $E$, while it is contained in $\Om$ by hypothesis. This implies that there exists some $x\in B_r(\gamma(t))\setminus E$, thus in particular there is some $x\in B_r(\gamma(t))\cap \partial E$, recalling that $B^*\subseteq E$. Consider now the connected component $S$ of $\partial E$ passing through $x$: we know that it is an arc of circle, with radius $r$, and with both endpoints in $\partial \Om$. Recall that a Cheeger set is always in the ``interior part'' of the arcs of circle forming its boundary or, in other words, the curvature of $\partial E$ is always positive (and equal to $1/r$) inside $\Om$. Since $t$ is very close to $t^*$ and since the endpoints of the arc $S$ must be outside $E$, thus outside both balls $B^*$ and $B_r(\gamma(t))$, which have the same boundary curvature as that of $S$, we infer that the length of $S$ is at least $\pi r - \eps$, where $\eps$ can be chosen arbitrarily small if $t$ and $t^*$ are close enough.\par

Notice now that $\partial E$ has only finitely many connected components with length greater than $\pi r -\eps$, then we can assume the existence of a sequence $t_n \searrow t^*$ as before, such that the same arc $S$ intersects the interior of each ball $B_r(t_n)$. Since $t_n$ is converging to $t^*$, the distance between $S$ and $\partial B^*$ must be zero, and then, also recalling Lemma~\ref{lemma:180}, we obtain that $S$ is exactly an arc of circle of length $\pi r$ inside $\partial B^*$.\par

Let us now take some $t^+\in (t^*,1)$, and consider the sets
\begin{align*}
E^-= \bigcup_{0<t<t^*} B_r(\gamma(t))\,, && 
E^+= \bigcup_{0<t<t^+} B_r(\gamma(t)) \,.
\end{align*}
It can be easily proven (and this will be explicitely done later in Section~\ref{section:strip}) that the following exact formulae hold
\begin{align*}
P(E^+)-P(E^-) = 2\ell\,, && |E^+\setminus E^-|= 2r\ell\,,
\end{align*}
being $\ell$ the length of $\gamma$ restricted to $(t^*,1)$. Calling $\widehat E=E \cup E^+$, we aim to show that $\widehat E$ is also a Cheeger set, which gives a contradiction with the maximality of $E$. If $E$ does not intersect $E^+\setminus E^-$, this is clear, because since $E$ is a Cheeger set then
\[
\frac{P(\widehat E)}{|\widehat E|} =\frac{P(E) +P(E^+)-P(E^-) }{|E| + |E^+|-|E^-|} = \frac{h(\Om)|E|+ 2\ell}{|E| + 2 r\ell}=h(\Om)\,,
\]
so $\widehat E$ is also a Cheeger set, a contradiction.\par
To conclude, suppose then that $F:=E \cap \big(E^+\setminus E^-\big)\neq \emptyset$, and notice that, as a consequence,
\begin{align}\label{pergip}
|\widehat E| = |E| + |E^+|-|E^-| - |F|\,, && 
P(\widehat E) = P(E) + P(E^+) - P(E^-) - P(F;E^+\setminus E^-)\,.
\end{align}
Since $\partial F$ has constant curvature $h(\Om)$ inside $E^+\setminus E^-$, while $\partial(E^+\setminus E^-)$ has curvature smaller than $h(\Om)$ by construction, we infer that
\[
P(F; E^+\setminus E^-) \geq \frac{P(F)}2\,.
\]
Finally, since by choosing $t^+$ close enough to $t^*$ we have that $|F|$ is arbitrarily small, hence by the isoperimetric inequality $P(F)/|F|$ is arbitrarily big, from~(\ref{pergip}) we directly obtain $P(\widehat E)/|\widehat E|\leq h(\Om)$, which gives again a contradiction.
\end{proof}

\begin{oss}\label{oss:movingball}\rm
The requirement in Lemma~\ref{lemma:movingball} of the maximality of $E$ can be dropped whenever the rolling ball remains at a positive distance from $\de\Om$. In this case, the previous proof ensures that the rolling ball will never intersect $\de E$, so that the claim follows.
\end{oss}

\section{Characterization of Cheeger sets in planar strips\label{section:strip}}

In recent years, there has been an increasing interest in the geometrical and spectral study of three-dimensional waveguides, or in their two-dimensional counterpart, the ``strips''. Basically, a waveguide is a generalization of a cylinder; more precisely, while a cylinder is obtained by taking a segment and considering the union of equal disks centered on the points of the segment and orthogonal to the segment itself, a waveguide is done in the very same way, just substituting the segment with a sufficiently regular curve. They are important in many applications, for instance for their optical properties. A number of interesting questions concern the spectral properties of the waveguides in the limit when they become extremely thin; for more information, the reader could see~\cite{DucExn1995,KreKri2005} and the references therein. In the two-dimensional case, a strip is constructed by taking a sufficiently regular curve and considering the union of equal segments centered on the points of the curve and orthogonal to the curve itself. In this sense, strips are generalized rectangles.\par\smallskip

Let us now give all the formal definitions concerning the strips. Let $\gamma: [0,L]\to \R^2$ be a ${\rm C}^{1,1}$ curve, with $\gamma(0)\neq \gamma(L)$, parametrized by arc-length. For every $t\in [0,L]$, we denote by $\sigma(t)$ the relatively open segment of length $2$ centered in $\gamma(t)$ and orthogonal to $\gamma'(t)$, and we define the set
\[
\strip:= \bigcup_{t\in (0,L)} \sigma(t)\,.
\]
Let us denote by $\Psi$ the obvious parametrization of $\strip$ on $(0,L)\times (-1,1)$; more precisely, calling $\nu(t)$ the normal vector whose direction is obtained by counter-clockwise rotating $\gamma'(t)$ of $90^\circ$, the map $\Psi: (0,L)\times (-1,1)\to \strip$ is given by
\[
\Psi(t,\rho) := \gamma(t) +\rho\nu(t)\,.
\]
Finally, we say that $\strip$ is an \emph{open strip of width $2$} if the map $\Psi$ is a ${\rm C}^{1,1}$ diffeomorphism: notice that this implies, in particular, that all the open segments $\sigma(t)$ are disjoint, hence in particular that the curvature of $\gamma$ is always at most $1$. The curve $\gamma$ is sometimes called the \emph{spinal curve} of the strip $\strip$.\par

It is possible to consider the case of the \emph{closed strips} (or \emph{annuli}), which roughly speaking correspond to the case when $\gamma(0)=\gamma(L)$; however, in this article we concentrate ourselves to the case of the open strips because the other case was already completely discussed in~\cite{KrePra2011}. It is also possible to consider the open strips of any width $2s$, where the segments $\sigma(t)$ have length $2s$ instead of $2$, but we prefer to fix the width to $2$ for simplicity of notations, since the general case can be trated by a trivial rescaling.\par
\begin{figure}[ht]
\centering
\includegraphics[scale=1]{cheegerfig-0}
\caption{A planar strip $\strip$.}
\label{fig:strip}
\end{figure}
The study of the behaviour of the strips in the Cheeger problem has been already started in~\cite{KrePra2011}, see also~\cite{PrSa14}. In particular, an interesting feature of the Cheeger problem for strips is that some of the properties which are valid in the convex case (recall the discussion of Section~\ref{sectconv}) are still valid for the strips, which are not convex. Another interesting fact is a two-sided bound on the Cheeger constant of strips when their length goes to infinity, found in the paper~\cite{KrePra2011} (see Theorem~\ref{teo:KrePra} below). Notice that, up to a rescaling, having fixed the width of the strips and letting their length go to infinity is completely equivalent to fix the length and consider the limit when they become infinitely thin, which is the physically relevant case, as said before.\par

The aim of the present article is to push forward the understanding of the properties of the convex case still valid for the strips, and to prove a refinement of the bounds on the Cheeger constant found in~\cite{KrePra2011} --basically, while the bounds there were at the first order, we obtain a second order kind of estimates, see Teorem~\ref{teo:asintotico} below. The following estimate is the one proved in~\cite{KrePra2011}.
\begin{teo}\label{teo:KrePra}
Let $\strip$ be a strip of length $L$ and width $2$. Then
\[
1 + \frac{1}{400\,L} \leq h(\strip) \leq 1 + \frac 2L\,.
\]
\end{teo}
We will be able to prove the following improvement of the above estimate.
\begin{teo}\label{teo:asintotico}
Let $\strip$ be an open strip of length $L$ and width $2$. Then 
\begin{align}\label{stripasymptotic}
h(\strip) = 1+ \frac{\pi}{2L} + O(L^{-2})\qquad \text{as }L\to +\infty\,.
\end{align}
\end{teo}
It is important to underline that the asymptotic estimate~\eqref{stripasymptotic} is optimal. The main brick to prove it is the following theorem, which is our main result concerning strips.
\begin{teo}\label{teo:unionedipalle2}
Let $\strip$ be an open strip of length $L\geq 9\pi/2$, and let $r=h(\strip)^{-1}$. Then there is a unique Cheeger set $E$ of $\strip$, which can be described as 
\begin{equation}\label{intergrafico}
E = \Psi\left(\{(t,s):\ 0< t< L,\ \rho^{-}(t)<s<\rho^{+}(t)\}\right)
\end{equation}
where $\rho^{+},\rho^{-}:[0,L]\to [-1,1]$ are two suitable continuous functions. Moreover, $E$ coincides with the union of all balls of radius $r$ contained in $\strip$, it is simply connected, and it can be obtained as the Minkowski sum $E = E_{r}+B_{r}$, where
\[
E_{r} = \{x\in \strip:\ \dist(x,\de\strip)\geq r\}
\]
is a set with Lipschitz boundary and positive reach $\reach(E_{r})\geq r$. Finally, the inner Cheeger formula
\begin{equation}\label{ICformula}
|E_{r}| = \pi r^{2}
\end{equation}
holds true.
\end{teo}

\begin{oss}\rm
We stress that the conclusions of Theorem~\ref{teo:unionedipalle2} (in particular, the fact that the Cheeger set $E$ is the union of all balls of radius $r$ contained in $\strip$, and that~\eqref{ICformula} holds true) are typical properties of the convex case, keep in mind Theorem~\ref{teo:kawlac}, but they are not satisfied by a generic planar domain. Two examples showing that no inclusion between Cheeger sets and unions of balls of radius $r$ is generally true, are given in the last section (Examples~\ref{ex:cheegerinunion} and~\ref{bow-tie}). Concerning the inner Cheeger formula~\eqref{ICformula}, there exists a star-shaped domain whose Cheeger set is the union of all included balls of radius $r$, but for which the formula fails (see Example~\ref{loose bow-tie}). 
\end{oss}

%We will notice that formula~(\ref{stripasymptotic}) actually comes straightforward from the results of Theorem~\ref{teo:unionedipalle2}. In turn, the proof of Theorem~\ref{teo:unionedipalle2} is based on two key lemmas.\par

The plan of the remaining of this section is the following. First, we prove a simple technical relation between perimeter and area of a strip and length of the corresponding spinal curve: in particular, we will see that perimeter and area of a strip depend only on the length of the spinal curve, not on the shape or curvature of the curve itself. Then, we will use this relation to give a straightforward proof of Theorem~\ref{teo:asintotico} by means of Theorem~\ref{teo:unionedipalle2}. After that, we will state and prove our two ``key lemmas''. And finally, we will conclude the section with the proof of Theorem~\ref{teo:unionedipalle2}.

\begin{prop}\label{prop:stripAP}
Let $\strip$ be a strip of length $L$ and width $2$. Then the area and the perimeter of $\strip$ are given, respectively, by
\begin{align}
\label{APstrip}
|\strip| = 2L\,, && P(\strip) = 2L+4\,.
\end{align}
\end{prop}
\begin{proof}
Let $\g$ be the spinal curve of $\strip$ and notice that, since $\gamma$ is parametrized by arc-length, the curvature $\kappa(t)$ is simply given by
\[
\gpp(t) = \kappa(t)\nu(t)\,,
\]
being $\nu(t)$ the normal vector to the curve $\gamma$ at $t$ obtained by counter-clockwise rotating $\gamma'(t)$ by $90^\circ$; as already noticed, the curvature $\kappa$ is always between $1$ and $-1$. Recall now that the map $\Psi$ is a ${\rm C}^{1,1}$ diffeomorphism between the rectangle $(0,L)\times (-1,1)$ and the strip $\strip$, being $\Psi(t,\rho) = \g(t) + \rho\cdot \nu(t)$. The Jacobian of $\Psi$ can then easily be calculated as
\[
J(t,\rho) = \big\|(\gp(t) + \rho\cdot \gpp(t)^{\perp})\wedge \nu(t)\big\| =
|1-\rho \kappa(t)|=1-\rho \kappa(t)\,,
\]
where the last equality comes from the fact that $|\rho|\leq 1$ and $|\kappa(t)|\leq 1$. Therefore, a simple change of variables gives us directly the following formula for the area of the strip $\strip$:
\[
|\strip| = \int_0^{L} \int_{-1}^1 1-\rho\kappa(t)\, d\rho\,dt = 2L\,.
\]
Instead, the perimeter is obtained adding the length of the ``left'' and ``right'' side of the boundary of the strip (that is, $\sigma(0)$ and $\sigma(L)$, each of which having length $2$) to the length of the ``top'' and ``bottom'' parts, corresponding to $\rho=\pm 1$, so we find
\[
P(\strip) = 4+\int_0^L J(t,1)+J(t,-1)\, dt = 4 + \int_0^L \big(1 - \kappa(t)\big)+\big(1 + \kappa(t)\big)\,dt= 4+2L\,,
\]
so~(\ref{APstrip}) is obtained and the proof is concluded.
\end{proof}
\begin{oss}\label{lenpar}
More generally, the above proof shows that if $\Gamma$ is any measurable subset of the spinal curve $\gamma$, and $G$ is the corresponding union of segments $\sigma(t)$ issuing from $\gamma(t)\in\Gamma$, we have  $|G| = 2\Hau^{1}(\Gamma),$ where $\Hau^{1}$ denotes the one-dimensional Hausdorff measure in $\R^{2}$. 
\end{oss}

\begin{proof}[Proof of Theorem~\ref{teo:asintotico}]
Let us start by estimating the area of the set $E_r$ defined in Theorem~\ref{teo:unionedipalle2}: let $x\in\strip$ be a point, and let $(t,\rho)\in (-L,L)\times (-1,1)$ be defined as $(t,\rho)=\Psi^{-1}(x)$. We clearly have $x\notin E_r$ if $|\rho|>1-r$, thus $E_r$ is contained in the strip of width $2(1-r)$ with the same spinal curve $\gamma$; by Proposition~\ref{prop:stripAP}, this implies
\[
|E_r| \leq 2(1-r)L\,;
\]
in particular, we get that $r<1$, because otherwise $E_r$ would be empty, and this would be against the claim of Theorem~\ref{teo:unionedipalle2}. On the other hand, take $x=\Psi(t,\rho)$ and assume that $|\rho|\leq 1-r$; by the definition of the strips, $x$ belongs to $E_r$ unless it has a distance smaller than $r$ from one of the two segments $\sigma(0)$ and $\sigma(L)$. If this is the case, then the whole segment $\sigma(t)$ has distance smaller than $2$ from either $\sigma(0)$ or $\sigma(L)$, since any point of $\sigma(t)$ has distance less than $2-r$ from $x$. If we call then
\[
\Gamma=\Big\{ \gamma(t):\, \exists \rho\in \big[-(1-r),1-r\big],\, \Psi(t,\rho)\notin E_r\Big\}\,,
\]
then the strip corresponding to $\Gamma$ in the sense of Remark~\ref{lenpar} has an area $2\H^1(\Gamma)$. Being this area concentrated in the two zones of points having distance less than $2$ from $\sigma(0)$ and from $\sigma(L)$, we deduce
\[
2\H^1(\Gamma) \leq 2\big( 4\pi + 8)\,.
\]
Summarizing, the strip corresponding to the whole spinal curve without $\Gamma$, and with width $2(1-r)$, is contained in $E_r$, hence recalling again Proposition~\ref{prop:stripAP} we finally derive
\[
|E_r| \geq 2(1-r)\Big(L -4\pi-8\Big) \,.
\]
Then, the inner Cheeger formula~\eqref{ICformula} gives
\[
2(1-r)\Big(L -4\pi-8\Big) \leq \pi r^{2} \leq 2(1-r)L\,,
\]
which finally implies~\eqref{stripasymptotic} by an elementary computation, recalling that $h(\strip)=1/r$.
\end{proof}

Let us claim and prove the two key lemmas, which will later be used for the proof of Theorem~\ref{teo:unionedipalle2}. The first lemma states that, if $E$ is a Cheeger set in $\strip$, and the length of $\strip$ is large enough, then any osculating ball to $\de E\cap \strip$ is entirely contained in $\strip$ (see Figure~\ref{fig:arcball}).
\begin{figure}[ht]
\centering
\includegraphics[scale=1]{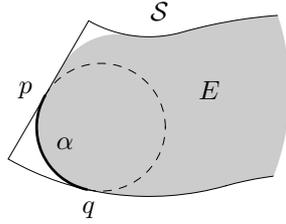}
\caption{The arc-ball property of a Cheeger set inside a strip.}
\label{fig:arcball}
\end{figure}

\begin{lemma}[Arc-ball property]\label{lemma1.1}
Let $E$ be a Cheeger set inside a strip $\strip$ of length $L\geq \frac{9\pi}2$. Set $r = h(\strip)^{-1}$. Then $\de E\cap \strip$ is non-empty, and for any circular arc $\alpha$ contained in $\de E\cap \strip$ the ball $B_{r}$, such that $\alpha\subseteq\de B_{r}$, is entirely contained in $\strip$.
\end{lemma}
\begin{proof}
We split the proof in three steps.
\medskip

\step{I}{$\strip$ is not Cheeger in itself, and $r<1$.}
It is immediate to observe that, if a set has a concave corner, then a small cut around it (of size $\eps$) decreases the perimeter of order $\eps$, and the area only of order $\eps^2$, hence for $\eps$ small enough the Cheeger ratio is decreased. This simple remark ensures that a Cheeger set can never contain a concave corner, thus in particular a set with concave corners cannot be Cheeger in itself. In particular, having four $90$-degrees corners, a strip is never Cheeger in itself (and since a strip is connected, this implies that $\partial E\cap \strip$ is not empty). Moreover, Theorem~\ref{teo:KrePra} ensures that $h(\strip)>1$, hence $r<1$.

\step{II}{Any connected component of $\partial E\cap \strip$ ``cuts a corner'' of $\strip$.}
Let $\alpha$ be a connected component of $\de E\cap \strip$, let $p,\,q \in \partial \strip$ its two endpoints, and call $B_r$ the ball whose boundary contains $\alpha$. In this step we are going to show that one of the two points $p$ and $q$ belongs to a lateral side (that is, $\sigma(0)$ or $\sigma(L)$), and the other point belongs either to the ``upper side'' or to the ``lower side'' --that is, $\Psi\big((0,L)\times \{1\}\big)$ and $\Psi\big((0,L)\times \{-1\}\big)$. To prove this claim, we have to exclude the following possibilities.\\
$\bullet$ Both $p$ and $q$ belong to $\sigma(0)$.\\
This is impossible. Indeed, recall that by Proposition~\ref{prop:CheegerGenProp}~(vii) the arc $\alpha$ meets $\partial\strip$ tangentially. Then, since $\sigma(0)$ is a segment, there cannot be a circle meeting twice $\sigma(0)$ in a tangential way. The same argument excludes also that both $p$ and $q$ belong to $\sigma(L)$.\\
$\bullet$ Both $p$ and $q$ belong to the upper side.\\
This is impossible. Indeed, as already noticed the arc $\alpha$ meets $\partial\strip$ tangentially; hence, if $p$ belongs to the upper side, say $p=\Psi(t,1)$, this implies that the center of $B_r$ belongs to the line orthogonal to $\partial\strip$ in $p$, so it must be in the segment $\sigma(t)$; here we just need to use that $r<2$ and that the arc $\alpha$ belongs to $\strip$. Now, if also $q$ belongs to the upper side, say $q=\Psi(t',1)$, for the same reason we discover that the center of $B_r$ belongs to the segment $\sigma(t')$. But since $\sigma(t)\cap\sigma(t')=\emptyset$, because all the segments are disjoint and $t\neq t'$ because $p$ and $q$ are distinct, this gives a contradiction. The same argument excludes also that both $p$ and $q$ belong to the lower side.\\
$\bullet$ $p$ belongs to the upper side, and $q$ to the lower one.\\
This is impossible. Indeed, the same argument as above implies that $p=\Psi(t,1)$ and $q=\Psi(t,-1)$; but then the distance between $p$ and $q$ would be $2$, while any two points in $\partial B_r$ have distance at most $2r<2$.\\
$\bullet$ $p\in\sigma(0)$ and $q\in\sigma(L)$.\\
This is impossible. Indeed, if the two endpoints of $\alpha$ belong to $\sigma(0)$ and $\sigma(L)$, then by continuity every segment $\sigma(t)$, for $0<t<L$, would intersect $\alpha$. But then every point of $\strip$ would have distance less than $2+r<3$ from the center of $B_r$, so the whole strip $\strip$ would be contained in the ball of radius $3$ having the same center as $B_r$. This would imply that $2L=|\strip| < 9\pi$, which is in turn ruled out by the assumption on $L$. The step is then concluded.

\step{III}{The whole ball $B_r$ is contained in $\strip$.}
In this last step we conclude the proof of the lemma. Thanks to Step~II, we can assume without loss of generality that $p\in\sigma(0)$ while $q$ belongs to the lower side. Hence, $q=\Psi(t,-1)$ for some $0<t<L$, and thus the center of $B_r$ is the point $\Psi(t,-1+r)$. Assuming by contradiction that $B_r$ is not contained in $\strip$, there must be some point $s\in B_r\cap \partial \strip$. The arguments of Step~II already ensure that $s$ cannot be contained in $\sigma(0)$, nor in $\sigma(L)$, so we must exclude that $s$ belongs to the upper or to the lower side. But in fact, if $s$ belongs to the upper side or to the lower side, then $s$ has distance less than $r$ from $\Psi(t,-1+r)$, thus less than $1$ from $\Psi(t,0)$, which is against the definition of strip. The proof is then concluded.
\end{proof}

The second lemma establishes a \emph{ball-to-ball} connectivity property of a generic strip, that is, the possibility of connecting two balls of radius $r$ that are contained in $\strip$ by rolling one of them towards the other, following a suitable path of centers with controlled curvature and preserving the inclusion in $\strip$ (see Figure~\ref{fig:ball2ball}).
\begin{figure}[ht]
\centering
\includegraphics[scale=1]{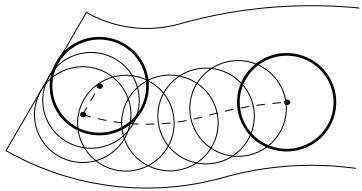}
\caption{The ball-to-ball property of a strip.}
\label{fig:ball2ball}
\end{figure}

\begin{lemma}[Ball-to-ball property]\label{lemma1.2}
If $B_r(x_{0})$ and $B_r(x_{1})$ are two balls of radius $r\leq 1$, both contained in a strip $\strip$, then there exists a piece-wise ${\rm C}^{1,1}$ curve $\beta:[0,1]\to \strip$ such that $\beta(0) = x_{0}$, $\beta(1) = x_{1}$, the curvature of $\beta$ is smaller than $r^{-1}$, and $B_{r}(\beta(t))\subseteq\strip$ for all $t\in (0,1)$.
\end{lemma}
\begin{proof}
We split again the proof into three steps. 

\step{I}{Centers of balls in $\strip$ with fixed $\rho$ and $r$ are projected onto arcs of the spinal curve.}
Let us fix $r\leq 1$ and $\rho\in (-1,1)$, and call for brevity $x_t=\Psi(t,\rho)$ for $0<t<L$. The goal of this step is to show that the set
\[
J:=\big\{t \in (0,L):\, B_r(x_t)\subseteq \strip \big\}
\]
is a closed interval. To begin, we observe that $J$ is clearly closed; moreover, it is admissible to assume $|\rho|\leq 1-r$, since otherwise $J$ is empty and there is nothing to prove. We observe now that, since $|\rho|\leq 1-r$, then every ball $B_r(x_t)$ has an empty intersection with the upper side $\Psi\big( (0,L)\times\{1\}\big)$, as well as with the lower side $\Psi\big( (0,L)\times\{-1\}\big)$. Therefore, by continuity, for each $t\in \partial J$ the boundary of the ball $B_r(x_t)$ must necessarily be tangent to either $\sigma(0)$ or $\sigma(L)$ (or both). We claim now that there cannot be two distinct $t_1<t_2$ in $\partial J$ such that the balls $B_r(x_{t_1})$ and $B_r(x_{t_2})$ are both tangent to $\sigma(0)$ (by symmetry, the same will be true for $\sigma(L)$): since for sure $t \notin J$ when $t$ is too close to $0$ or $L$, the fact that $J$ is a closed segment will follow at once as soon as we show this claim.\par

Assume then the existence of $t_1<t_2$ against the claim. Then, the segment connecting $x_{t_1}$ and $x_{t_2}$ is parallel to $\sigma(0)$, and as a consequence there must be some $\bar t\in (t_1,t_2)$ such that the direction $\gamma'(\bar t)$ of the curve $\gamma$ at $\gamma(\bar t)$ is parallel to $\sigma(0)$. In other words, the segment $\sigma(\bar t)$ is orthogonal to the segment $\sigma(0)$. Now, observe that all segments $\sigma(t)$ are disjoint, both endpoints of any $\sigma(t)$ are by definition outside the ball $B_r(x_{t_2})$, and $\sigma(0)$ is tangent to $B_r(x_{t_2})$. Then, an immediate geometric argument implies that the segments $\sigma(t)$, for $t$ varying between $0$ and $\bar t$, sweep strictly more than one half of the ball $B_r(x_{t_2})$. As a consequence, there exists some $t'\in [0,\bar t]$ such that $\sigma(t')$ contains the point $x_{t_2}$. And finally, this is impossible, because $x_{t_2}$ only belongs to $\sigma(t_2)$ and $t_2> \bar t \geq t'$.

\step{II}{The case of two balls tangent to $\sigma(0)$.}
Let $x,\, y\in \strip$ be two points such that both balls $B_r(x)$ and $B_r(y)$ are contained in $\strip$ and tangent to $\sigma(0)$. We claim that the (open) convex envelope $K$ of the two balls $B_r(x)$ and $B_r(y)$ is entirely contained in $\strip$. We prove the claim by contradiction, assuming that $\partial\strip\cap K$ is not empty. By construction, it is clear that $\sigma(0)$ does not intersect $K$; moreover, since $\sigma(L)$ does not intersect $B_r(x)$ nor $B_r(y)$, then it is impossible that $\partial \strip\cap K$ consists only of points of $\sigma(L)$. As a consequence, there must be points of $\partial \strip\cap\partial K$ which are not in $\sigma(0)\cup \sigma(L)$, hence which are either in the upper side $\partial^+\strip$ or in the lower side $\partial^-\strip$. Since we can assume that both endpoints of $\sigma(0)$ are a strictly positive distance apart from $K$ (because otherwise the claim is immediate), there is a strictly positive $t_0<L$ such that one of the endpoints of $\sigma(t_0)$ belongs to $\partial K\setminus \sigma(0)$, and $t_0$ is the smallest number for which this happens. Without loss of generality, let us assume that this endpoint is the upper one, that is, $\Psi(t_0,1)\in \partial K$. Notice that this point does not belong to $\sigma(0)$, nor to the boundaries of $B_r(x)$ and $B_r(y)$, hence it belongs to the segment of $\partial K$ which is parallel to $\sigma(0)$ but not intersecting $\sigma(0)$.\par

Let us now call $K^+$ the biggest bounded set in $\R^2$ whose boundary is contained in the union of $\partial K$, $\sigma(0)$, and the curve $t\mapsto \Psi(t,1)$ with $0<t<t_0$. Observe that $K^+\supsetneq K$, and that the whole curve $t\mapsto \Psi(t,1)$ with $0<t<t_0$ is part of $\partial K^+$; observe also that the curve $t\mapsto \Psi(t,-1)$ for $0\leq t\leq t_0$ does not intersect $K^+$, by the minimality of $t_0$. Notice now that the direction of the segment $\sigma(t_0)$ has a strictly negative component in the direction of $\sigma(0)$: this comes again by the minimality of $t_0$, since the curve $t\mapsto \Psi(t,1)$ must enter in the stadium $K$ from the straight side which is not contained in $\sigma(0)$. As a consequence, the segment $\sigma(t_0)$ must lie entirely inside $K^+$: indeed, in the point $\Psi(t_0,1)$ the segment is pointing inside $K^+$ by construction, and the argument above about the direction implies that it could exit from $K^+$ only either at a point of $\sigma(0)$ --and this is impossible because $\sigma(t_0)$ and $\sigma(0)$ do not intersect-- or at a point of $\partial^+\strip \cap \partial K^+$ --and this is impossible because $\sigma(t_0)$ and $\partial^+\strip$ cannot intersect. Summarizing, we have proved that the whole segment $\sigma(t_0)$ is inside $K^+$, hence in particular $\Psi(t_0,-1)$ is $K^+$; and finally, this gives the required contradiction because we proved above that $t\mapsto \Psi(t,-1)$ for $0\leq t\leq t_0$ does not intersect $K^+$.

\step{III}{Conclusion.}
The conclusion now easily follows from steps~I and~II. Given any two balls of same radius $r\leq 1$ contained inside the strip, by Step~I we can move both with constant distance from $\partial^+\strip$ until they become tangent to $\sigma(0)$; then, we can connect these two balls parallel to $\sigma(0)$ thanks to Step~II. The corresponding curve $\beta$ is clearly made by three ${\rm C}^{1,1}$ pieces, and by construction each piece has curvature smaller than $r^{-1}$.
\end{proof}

With these two lemmas at hand, we can finally prove Theorem~\ref{teo:unionedipalle2} and conclude this section.

\begin{proof}[Proof of Theorem~\ref{teo:unionedipalle2}]
For simplicity, we divide the proof in some steps.
\step{I}{The functions $\rho^\pm$ such that~(\ref{intergrafico}) holds.}
Let us start by defining the functions $\rho^\pm$. First of all, Lemma~\ref{lemma1.1} ensures the existence of arcs of circle contained in $\partial E\cap\strip$, each one associated with a corner of $\strip$ in the sense of Step~II of the proof of Lemma~\ref{lemma1.1}. As we already pointed out in that proof, each of the four corners of $\strip$ must be ruled out of $E$, and this means that there are at least four arcs in $\partial E\cap \strip$. We want to show that they are actually exactly four, or in other words that the four corners of $\strip$ are in a one-to-one correspondence with the connected components of $\partial E\cap \strip$. To do so, let $\alpha$ be such an arc, and let $B_r(x)$ be the ball which contains $\alpha$ as a part of its boundary: by Step~II of the proof of Lemma~\ref{lemma1.1}, we know that exactly one endpoint of $\alpha$ belongs either to the upper side $\partial^+\strip$ or to the lower side $\partial^-\strip$. If an endpoint of $\alpha$ belongs to $\partial^+\strip$, then we know that the center $x$ of $B_r(x)$ must be $x=\Psi(t,1-r)$ for some $t\in (0,L)$; conversely, if an endpoint of $\alpha$ belongs to $\partial^-\strip$, then $x=\Psi(t,-1+r)$ for some $t\in (0,L)$. However, Step~I of the proof of Lemma~\ref{lemma1.2} implies that there is exactly a single $x$ of the form $x=\Psi(t,1-r)$ such that the ball $B_r(x)$ is tangent to $\sigma(0)$, and exactly another one such that the ball is tangent to $\sigma(L)$. This ensures that there is exactly a single arc ruling out each of the four corners of $\strip$.\par

As a consequence, we know that $\partial E\cap \strip$ is made by four arcs of circle $A^+_l$, $A^+_r$, $A^-_r$ and $A^-_l$, connecting respectively the left side with the upper side, the upper with the right, the right with the bottom, and the bottom with the left. Therefore, $\partial E$ is the union of these four arcs, plus two segments $\Sigma_l$ and $\Sigma_r$ respectively in $\sigma(0)$ and $\sigma(L)$, plus two curves $\Gamma^+$ and $\Gamma^-$ respectively in $\partial^+\strip$ and $\partial^-\strip$. We can define the ``upper boundary'' $\partial^+E$ of $E$, and the ``lower boundary'' $\partial^- E$ of $E$ as
\begin{align*}
\partial^+E:= A^+_l \cup \Gamma^+\cup A^+_r\,, && \partial^-E:= A^-_l \cup \Gamma^-\cup A^-_r\,.
\end{align*}
We now claim that for every $0<t<L$ there is exactly one number $\rho^+(t)\in [-1,1]$, and exactly one number $\rho^-(t)\in [-1,1]$, such that $\Psi\big(t,\rho^\pm(t)\big)\in \partial^\pm E$. Indeed, otherwise there should be a segment $\sigma(t)$ which intersects twice one of the four arcs, say $A^+_l$, and then by continuity there should be another segment $\sigma(t')$ which is tangent to the arc $A^+_l$: this is impossible, because then the whold segment $\sigma(t')$ would be in the upper-left connected component of $\strip\setminus E$, and this is absurd because $\Psi(t',-1)$ cannot be in that component. It readily follows that the functions $\rho^\pm$ are continuous, and that they fulfill the property~(\ref{intergrafico}); the fact that $E$ is simply connected is then immediate. Moreover, the construction ensures that the four arcs $A^\pm_{l,\,r}$ only depend on $\strip$, not on $E$, and thus there is actually a unique Cheeger set.

\step{II}{$E$ is the union of balls of radius $r$ contained in $\strip$ and $E=E_r+B_r$.}
Let us now call $U$ the union of all balls of radius $r$ contained in $\strip$, and let us show that $E = U$. First of all, let $B_r$ be any ball of radius $r$ contained in $E$ (we know that such a ball exists, for instance any of the balls containing one of the four arcs $A^\pm_{l,\,r}$ have this property by Lemma \ref{lemma1.1}). Then let $B'_r$ be any other ball, still of radius $r$, contained in $\strip$. Lemma~\ref{lemma1.2} ensures that we can roll the ball $B_r$ in $\strip$ until it covers $B'_r$. Then, the ``rolling ball'' Lemma~\ref{lemma:movingball} implies that all the intermediate balls, and in particular $B'_r$, are contained in $E$; notice that we can apply Lemma~\ref{lemma:movingball} because $E$, being the unique Cheeger set as shown in the previous step, is in particular the maximal one. Hence, we have proved that $U\subseteq E$.\par

Conversely, notice that the assumption on $L$ gives, thanks to Theorem~\ref{teo:KrePra}, that $h(\strip)<2$, thus $r>1/2$. Hence, every point of $\strip$ has distance less than $2r$ from $\partial\strip$, and so every point of $E$ has distance less than $2r$ from $\partial E$. Therefore, to check that $E\subseteq U$, it is  sufficient to prove that, for each point $p\in \partial E$, the ball with radius $r$ tangent to $\partial E$ in $p$ (from the same side of $E$) is entirely contained in $\strip$.\par

If $p$ belongs to one of the four arcs of $\partial E$, this is ensured by Lemma~\ref{lemma1.1}, while if $p$ belongs to $\sigma(0)$ or $\sigma(L)$, this is an immediate consequence of Step~II of the proof of Lemma~\ref{lemma1.2}. Assume then by contradiction the existence of a point $p$ in $\partial^+\strip\cap \partial E$ such that the above-mentioned ball is not contained in $\strip$; by construction, this means that the circle of radius $r$ tangent in $p$ to $\partial^+\strip$ has two points which both belong to $\sigma(0)$ or both to $\sigma(L)$; since this happens exactly with one of the two segments, by the assumption on $L$, we assume that the segment is $\sigma(0)$. Since $p\in\partial E$, then by Step~I it is ``after'' the arc $A^+_l$: this means that, if we call $t'$ such that $p=\Psi(t',1)$, and $t''$ such that $\Psi(t'',1)$ is the extreme of the arc $A^+_l$ which belongs to $\partial^+\strip$, then $t'>t''$. Let us now ``roll'' the ball, that is, let us consider all the balls $B_t$ of radius $r$ which are tangent to $\partial^+\strip$ at points $\Psi(t,1)$ for $t>t'$: by continuity, there exists some $t'''$ such that this ball is tangent to $\sigma(0)$, before reaching the arc $A^+_r$. By construction, this ball is inside $\strip$ and tangent to $\partial\strip$ both on $\partial^+\strip$ and on $\sigma(0)$; moreover, since $t'''>t'>t''$, this ball does not coincide with $B_{t''}$ (which is the ball containing the arc $A^+_l$ in its boundary). Hence, there are two distinct balls, namely $B_{t''}$ and $B_{t'''}$, which are contained in $\strip$ and tangent to $\partial\strip$ both at the left and at the upper side. To conclude our claim, then, we need to show that this is impossible. We can now apply Step~II of Lemma~\ref{lemma1.2}, since $B_{t'''}$ is by construction in the stadium $K$: since $B_{t'''}$ is tangent to $\partial^+\strip$ at some point, but $\partial^+\strip$ cannot enter in the interior of $K$, we deduce that the point of tangency is the opposite point in $B_{t'''}$ to the point of tangency with $\sigma(0)$. This means that the segment $\sigma(t''')$ is orthogonal to $\sigma(0)$, and since the distance between the two opposite points is $2r<2$ this implies that $\sigma(0)$ and $\sigma(t''')$ intersect each other, which is the desired contradiction. Summarizing, we have proved that for every point $p\in\partial E$ the ball of radius $r$ tangent to $\partial E$ in $p$ is entirely contained in $E$, and as said above this implies that $E\subseteq U$, hence finally the equality $E=U$ is established. Moreover, since by definition $U=E_r+B_r$, the fact that $E=E_r+B_r$ follows.

\step{III}{$E_r$ has Lipschitz boundary and positive reach $\reach(E_r)\geq r$.}
In this step we consider the set $E_r$, and we show that it has Lipschitz boundary and positive reach $\reach(E_r)\geq r$. First of all, we can find a more or less explicit formula to describe the set $E_r$. More precisely, let $x=\Psi(t,\rho)$ be a point of $\strip$. The point $x$ surely does not belong to $E_r$ if $|\rho|>1-r$; on the other hand, if $|\rho|\leq 1-r$, then $x$ belongs to $E_r$ unless it has distance less than $r$ from one of the two sides $\sigma(0)$ and $\sigma(L)$. Suppose then that a point $x=\Psi(t,\rho)$ has distance less than $r$ from $\sigma(0)$ but $|\rho|\leq 1-r$: this implies that the point of $\sigma(0)$ minimising the distance from $x$ is in the interior of $\sigma(0)$, because since $|\rho|\leq 1-r$ the point $x$ has distance greater than $r$ from the upper and lower side of $\partial\strip$. Let us then take the segment $S_0$ parallel to $\sigma(0)$ at a distance $r$, in the direction of $\strip$: $x$ does not belong to $E_r$ if it is between $\sigma(0)$ and $S_0$; similarly, the point $x=\Psi(t,\rho)$ with $|\rho|\leq 1-r$ does not belong to $E_r$ if it is between $\sigma(L)$ and $S_L$, the latter being the segment parallel to $\sigma(L)$ having distance $r$, in the direction of $\strip$. Notice that we can limit ourselves to rule out the points between $\sigma(0)$ (resp. $\sigma(L)$ and the segment $S_0$ (resp. $S_L$) with distance $r$ \emph{in the direction of $\strip$}: indeed, if $x$ has distance less than $r$ from $\sigma(0)$ but it is on the other side, then in particular it has also distance less than $r$ from $\sigma(L)$, and vice versa. Summarizing, also keeping in mind Step~I of the proof of Lemma~\ref{lemma1.2}, we can write
\[
\partial E_r = \Gamma^+ \cup \Gamma^- \cup S_l \cup S_r\,,
\]
where $\Gamma^+$ and $\Gamma^-$ are two arcs, respectively contained in the curve $t\mapsto \Psi(t,1-r)$ and $t\mapsto \Psi(t,-1+r)$, while $S_l$ and $S_r$ are two segments, respectively contained in the segments $S_0$ and $S_L$ defined above. Thanks to the properties of $\Psi$, we have then already that $E_r$ has Lipschitz boundary.\par

We must now check that the reach of $E_r$ is at least $r$; in other words, since $E=E_r + B_r$, for every point $x$ of $E\setminus E_r$ we have to check that there is a single point of minimal distance from $x$ in $E_r$. To do so, it is convenient to subdivide $E\setminus E_r$ in eight regions, as Figure~\ref{fig:stripdecomp} shows. The region $R_+$ is made by all points $\Psi(t,\rho)$ with $\rho>1-r$, and being $t\in (0,L)$ such that $\Psi(t,1-r)\in \Gamma^+$. Analogously, the generic point of $R_-$ is $\Psi(t,\rho)$ with $\rho<-1+r$ if $\Psi(t,-1+r)\in \Gamma^-$. The regions $D^\pm_{l,\,r}$ are the four circular sectors corresponding to the arcs $A^\pm_{l,\,r}$. And finally, the regions $Q_l$ and $Q_r$ are two rectangles, having as parallel sides the segment $S_l$ and the parallel subsegment of $\sigma(0)$, and the segment $S_r$ and the parallel subsegment of $\sigma(L)$ respectively. The fact that this is actually a subdivision of $E\setminus E_r$ comes readily from the construction and from the two key Lemmas~\ref{lemma1.1} and~\ref{lemma1.2}.\par
\begin{figure}[ht]
\centering
\includegraphics[scale=1]{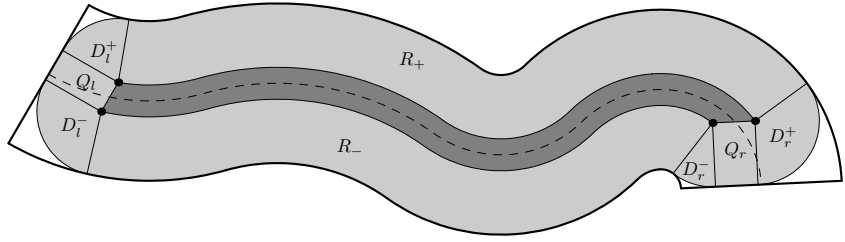}
\caption{The decomposition of the Cheeger set. The inner Cheeger set $E_{r}$ is colored in dark grey. One can see the eight regions of the decomposition of $E\setminus E_{r}$ colored in light grey.}
\label{fig:stripdecomp}
\end{figure}
Let us then take a generic point $x\in E\setminus E_r$, and let us show that it admits a unique point of minimal distance in $\partial E_r$. If $x\in R_+$, then $x=\Psi(t,\rho)$ for some $\rho>1-r$. Since any point $x'=\Psi(t',\rho')\in \partial E_r$ must satisfy $\rho'\geq 1-r$, then $|x-x'|\geq \rho-(1-r)$ for every $x'\in \partial E_r$, with strict inequality if $t'\neq t$ by the definition of strip; on the other hand, by definition $y=\Psi(t,1-r)$ belongs to $\partial E_r$, and $|x-y|=\rho-(1-r)$, so $y$ is the unique point of minimal distance in $\partial E_r$ from $x$. The same argument of course works if $x\in R_-$.\par
Assume now that $x\in Q_l$ (the fully analogous argument will work for $Q_r$). By construction, every point of $\partial E_r$ has distance greater than $r$ from $\sigma(0)$, with equality only for points of $S_l$. Thus, every point of the rectangle $Q_l$ has as unique point of minimal distance from $\partial E_r$ its orthogonal projection on $S_l$.\par
Finally, let $x$ be in the upper-left circular sector $D^+_l$ (and as usual, the analogous argument will work for the other three sectors). Then, by construction and immediate geometric considerations, it is clear that the center of the sector, which belongs to $\partial E_r$, is the unique point of minimal distance from $x$ (here we use again that $\Gamma^+$ is an arc, by Step~I of Lemma~\ref{lemma1.2}). Hence, we have proved the uniqueness of the minimizer in every possible case, and this step is concluded.

\step{IV}{The inner Cheeger formula~(\ref{ICformula}) holds true.}
Thanks to Step~III, we can apply Proposition~\ref{steinerformulas} to the set $A=E_r$, with $\rho=r$. Since then $A^r=A+B_r=E_r+B_r=E$, Steiner's formulae~(\ref{Asteiner}) and~(\ref{Psteiner}) read as
\begin{align*}
|E| = |E_r| + r P(E_r) + \pi r^2\,, && P(E)=P(E_r)+2\pi r\,.
\end{align*}
Recalling that $E$ is a Cheeger set, hence $P(E)/|E|=h(\strip)=1/r$, we deduce
\[
r = \frac{|E|}{P(E)} = \frac{|E_r| + r P(E_r) + \pi r^2}{P(E_r)+2\pi r}\,,
\]
from which one readily derives $|E_r|=\pi r^2$, that is, the validity of~(\ref{ICformula}) is established and the proof is concluded.
\end{proof}

\section{Some planar examples\label{sect:examples}}

This last section is devoted to collect some examples of non-convex planar domains, whose Cheeger sets have particular properties; basically, for most of the standard properties of the convex planar domains, that we have generalized for the case of the strips, we show non-convex domains which are not strips, and for which these properties are not valid.

The first example is a domain $G$, whose Cheeger set is strictly contained in the union of balls of radius $r = h(G)^{-1}$ that are contained in $G$. 
\begin{example}[\cite{KawLac2006}]\label{ex:cheegerinunion}\rm 
Let $G$ be the union of two disjoint balls $B_{1}$ and $B_{\frac 23}$, of radii $1$ and $\frac 23$ respectively (see Figure~\ref{fig:cheegerinunion}). One has $\frac{P(G)}{|G|} = \frac{30}{13} >2$. It is not difficult to check that the Cheeger set $E$ of $G$ coincides with $B_{1}$, hence $h(G) = 2$. However, $G$ coincides with the union of all balls of radius $r = h(G)^{-1} = \frac 12$ contained in $G$, which is therefore strictly larger than $E$. 
\begin{figure}[ht]
  \centering
  \includegraphics[scale=1]{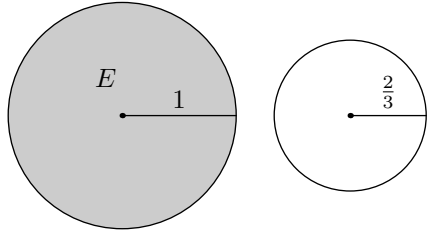}
  \caption{A union of two disjoint balls $B_{1}$ and $B_{\frac 23}$, whose Cheeger set $E$ coincides with the largest ball $B_{1}$.}
  \label{fig:cheegerinunion}
\end{figure}
\end{example}

The next example shows a Cheeger set $\bowtie$ \emph{strictly containing} the union of all balls of radius $h(\bowtie)^{-1}$ contained in $\bowtie$. This example and the one depicted in Figure~\ref{fig:cheegerinunion} show that, in general, no inclusion holds between a Cheeger set of $\Om$ and the union of all balls of radius $r = h(\Om)^{-1}$ contained in $\Om$.
\begin{example}[Bow-tie]\label{bow-tie}\rm
Let us consider a unit-side equilateral triangle $T$, as in Figure~\ref{fig:bowtie}, together with its Cheeger set $E_{T}$ (depicted in grey). Then, cut $T$ with the vertical line tangent to $E$ and reflect the portion on the left to the right, as shown in the picture. This produces a bow-tie $\bowtie$.
\begin{figure}[ht]
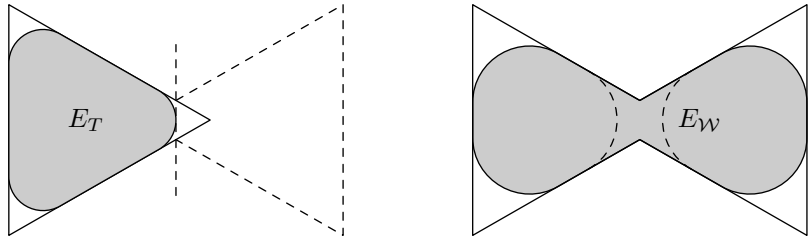

  \centering
  \includegraphics[scale=1]{cheegerfig-3}
  \qquad\qquad
  \includegraphics[scale=1]{cheegerfig-4}
  \caption{The construction of the bow-tie $\mathcal W$ (left) and the Cheeger set $E_{\mathcal W}$ in the bow-tie (right). Notice that the region between the two dashed lines in the picture on the right is the difference between the Cheeger set $E_{\mathcal W}$ and the (strictly smaller) union of all balls of radius $r$ included in $\mathcal W$.}
  \label{fig:bowtie}
\end{figure}
Let now $E_{\bowtie}$ be a Cheeger set inside $\bowtie$. By the $2$-symmetry of $\bowtie$ one can infer the $2$-symmetry of $E_{\bowtie}$. On the other hand, $E_{\bowtie}$ cannot have a connected component $F$ completely contained in $T$, since otherwise $F$ would be Cheeger inside $\bowtie$ and, at the same time, it would coincide with $E_{T}$. But then $E_{T}\cup E_{T}'$ (denoting by $E_{T}'$ the reflected copy of $E_{T}$ with respect to the cutting line) would be Cheeger in $\bowtie$, which is not possible since $\de (E_{T}\cup E_{T}')\cap \bowtie$ is not everywhere smooth, as it should be according to Proposition~\ref{prop:CheegerGenProp}. Being necessarily $\de E_{\bowtie}\cap \bowtie$ equal to a finite union of circular arcs, all with the same curvature $h(\bowtie)$, it is not difficult to rule out all possibilities except the one in which $\de E_{\bowtie} \cap \bowtie$ is composed by four congruent arcs, one for each concave corner in the boundary of $\bowtie$. Moreover one has the strict inequality $h(\bowtie)<h(T)$, therefore the union of all balls of radius $h(\bowtie)^{-1}$ contained in $\bowtie$ does not contain $E_{\bowtie}$ (indeed, some small region around the two concave corners cannot be covered by those balls).  
\end{example}

The next example is obtained as a slight variation of Example~\ref{bow-tie}. In this case, the resulting Cheeger set is simply connected, while the inner Cheeger set is disconnected. As a result, we derive the impossibility for the inner Cheeger formula~\eqref{ICformula} to hold.
\begin{example}[Loose bow-tie]\label{loose bow-tie}\rm
Take the bow-tie $\bowtie$ constructed in the previous example and vertically move the two concave corners a bit far apart. By the continuity of the Cheeger constant (see~\eqref{hc}) we infer the existence of some minimal vertical displacement of the two corners, such that the Cheeger set $\widetilde E$ in the modified bow-tie $\widetilde\bowtie$ actually coincides with the union of all balls of radius $r = h(\widetilde \bowtie)^{-1}$. This corresponds to the situation represented in Figure~\ref{fig:loosebowtie}. 
\begin{figure}[ht]
\centering
\includegraphics[scale=1]{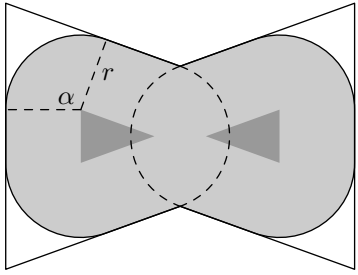}
\caption{A loose bow-tie for which the inner Cheeger formula does not hold.}
\label{fig:loosebowtie}
\end{figure}
It is then easy to check that the formula $|\widetilde E_{r}| = \pi r^{2}$ does not hold in this case, essentially because the inner Cheeger set $\widetilde E_{r}$ (depicted in dark grey) does not satisfy $\reach(\widetilde E_{r}) \geq r$. We also notice that, while the Cheeger set $\widetilde E$ is connected, the inner Cheeger set $\widetilde E_{r}$ is disconnected. Finally, one can easily check that the true formula, that is satisfied by the inner Cheeger set in this case, is
\[
|\widetilde E_{r}| = 2\alpha r^{2}>\pi r^{2}\,,
\]
where $\alpha$ is the angle depicted in Figure~\ref{fig:loosebowtie}.
\end{example}

Before getting to the last examples, we recall a result of generic uniqueness for the Cheeger set inside a domain $\Om\subseteq\R^{n}$, proved in~\cite{CasChaNov2010}:
\begin{teo}[\cite{CasChaNov2010}]\label{teo:genericunico}
Let $\Om\subseteq\R^{n}$ be any bounded open set, and let $\e>0$ be fixed. Then there exists an open set $\Om_{\e}\subseteq\Om$, such that $|\Om\setminus \Om_{\e}|<\e$ and the Cheeger set of $\Om_{\e}$ is unique.
\end{teo}
\begin{proof}[Idea of proof]
Let $E$ be a minimal Cheeger set of $\Om$, and let $\om_{\e}$ be a relatively compact, open subset of $\Om$ with smooth boundary, such that $|\Om\setminus \om_{\e}|<\e$. Define $\Om_{\e} = E \cup \om_{\e}$, then by an application of the strong maximum principle for constant mean curvature hypersurfaces one can show that $E$ is the unique Cheeger set of $\Om_{\e}$.
\end{proof}
\begin{example}[\cite{KawLac2006}]\label{ex:duecheeger}\rm 
Figure~\ref{fig:duecheeger} shows a simply connected domain consisting of two congruent squares connected by a small strip. Both the left and the right square with suitably rounded corners are Cheeger sets, and their union is the maximal Cheeger set of the domain.  
\begin{figure}[ht]
  \centering
  \includegraphics[scale=1.4]{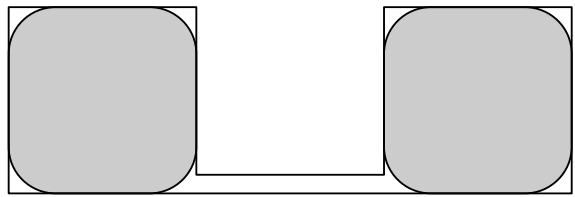}
  \caption{A simply connected domain whose Cheeger set is not unique.}
  \label{fig:duecheeger}
\end{figure}
\end{example}
A more sophisticated example of non-uniqueness, where the Cheeger sets are more than countably many, is constructed below; we point out that a similar example was numerically discussed by E. Parini in his master degree thesis~\cite{Parini_tesilaurea2006}.
\begin{example}[Pinocchio]\label{ex:pinocchio}
\rm 
Let $\pino_{\theta}$ be the union of a unit disc $B_{1}$ centered at $(0,0)$ and a disc of radius $r=\sin \theta$ and center $(\cos \theta,0)$, where $\theta\in(0,\pi/2)$ will be chosen later. The perimeter of $\pino_{\theta}$ is 
\[
P(\theta) = 2(\pi -\theta) + \pi \sin\theta, 
\]
while its area is
\[
A(\theta) = (\pi - \theta) + \sin\theta\, \cos\theta +\frac{\pi\sin^{2}\theta}{2}\,.
\]
\begin{figure}[ht]
\centering
\includegraphics[scale=1]{cheegerfig-9}
\caption{The set $\pino(\theta)$.}
\label{fig:pinocchio1}
\end{figure}
Given $0\leq \alpha \leq \pi/2 - \theta$ we define the subset $\pino(\theta,\alpha)$ of $\pino(\theta)$ as the union of $B_{1}$ and a disc of radius $r(\theta,\alpha) = \frac{\sin \theta}{\cos \alpha}$ and center $(\cos\theta - r(\theta,\alpha) \sin \alpha, 0)$. Clearly $\pino(\theta,\alpha)\subseteq\pino(\theta)$ for all $\alpha\in [0,\pi/2 - \theta]$. Moreover $\pino(\theta) = \pino(\theta,0)$ and $B_{1} = \pino(\theta,\pi/2-\theta)$. Notice also that the boundary of $\pino(\theta,\alpha)$ is made of two circular arcs meeting at the points $(\cos \theta, \pm \sin\theta)$. The perimeter of $\pino(\theta,\alpha)$ is 
\[
P(\theta,\alpha) = 2(\pi -\theta) + (\pi -2\alpha)\,\frac{\sin\theta}{\cos\alpha}, 
\]
while the area is
\[
A(\theta,\alpha) = (\pi - \theta) + \sin\theta\, \Big(\cos \theta - \sin\theta\,\tan\alpha \Big)
+\frac{\sin^2\theta}{\cos^2\alpha}\, \big(\pi/2 - \alpha\big)\,.
\]
Owing to Proposition \ref{prop:CheegerGenProp}(vii) and Lemma \ref{lemma:180}, it is not difficult to deduce that, for every $\theta\in (0,\pi/2)$, a Cheeger set in $\pino(\theta)$ must be of the form $\pino(\theta,\alpha)$ for some (in principle, not necessarily unique) $\alpha\in [0,\pi/2- \theta]$. Indeed, the boundary of any Cheeger set inside $\partial\pino(\theta)$ must be done by arcs of circle, which can only encounter $\partial \pino(\theta)$ tangentially, if the ``meeting point'' --that is, the point where an arc of circle internal to $\pino(\theta)$ reaches the boundary of $\pino(\theta)$-- is a point where $\partial\pino(\theta)$ has a tangent. This immediately ensures that no meeting point can exist, except possibly for the two corner points. But then, either $\pino(\theta)$ is Cheeger in itself, and then there is no meeting point, or the internal boundary of the Cheeger set must be an arc of circle connecting the two corner points. In other words, the possible Cheeger sets in $\pino(\theta)$ are only the sets $\pino(\theta,\alpha)$, as claimed.\par

Let us now look for the existence of $\theta_{0}\in (0,\pi/2)$ such that $\pino(\theta_0)$ is Cheeger in itself: a necessary but not sufficient condition is of course that
\[
\frac{P(\theta_{0},0)}{A(\theta_{0},0)} = \frac{1}{r(\theta,0)}= \frac{1}{\sin \theta_0}\,,
\]
that is, $\theta=\theta_0$ must be a solution of
\begin{equation}\label{thetavincolo}
2(\pi-\theta)\sin\theta +\frac{\pi}{2}\sin^{2}\theta - (\pi - \theta) - \sin\theta\cos\theta = 0
\end{equation}
We claim that there exists a unique $\theta_0\in (0,\pi/2)$ such that~\eqref{thetavincolo} is satisfied. Indeed, setting $g(\theta)$ equal to the left hand side of~\eqref{thetavincolo}, we have $g(0) = -\pi$ and $g(\pi/2) = \pi$, and for every $\theta\in(0,\pi/2)$ it holds
\[
g'(\theta) = 2(\pi-\theta)\cos\theta +\sin\theta\big(2\sin \theta +\pi \cos\theta -1\big)> 0\,,
\]
which proves our claim; a numerical approximation gives $\theta_{0}\simeq 0.531$. We set for brevity $\pino_{0} = \pino_{\theta_{0}}$, and we aim to prove that this set is Cheeger in itself (recall that the validity of the above equation was only a necessary, but not sufficient condition).\par

Assume then by contradiction that $\pino_0$ is not Cheeger in itself: as discussed above, this means that a Cheeger set in $\pino_0$ must be a set of the form $\pino(\theta_0,\alpha)$ for some $\alpha\in (0,\pi/2-\theta_0]$. We will now show that this is impossible.\par

To do so, we only have to check that $P\big(\pino(\theta_0,\alpha)\big)\sin\theta_0> A\big(\theta_0,\alpha\big)$ for every $\alpha\in (0,\pi/2-\theta_0]$. Writing down this inequality, and using~(\ref{thetavincolo}), we readily reduce to
\[
\frac \pi 2 \big( 1 -2\cos\alpha+\cos^2\alpha) < \alpha (1-2\cos\alpha) +\sin\alpha\cos\alpha\,,
\]
which in turn is a trigonometric inequality that can be elementary verified. We have thus proved that the set $\pino_0$ is Cheeger in itself.\par

\begin{figure}[ht]
\centering
\includegraphics[scale=1]{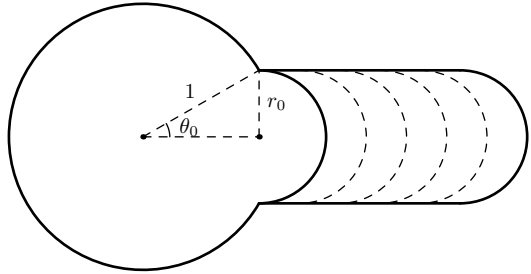}
\caption{The one-parameter family of Cheeger sets.}
\label{fig:pinocchio2}
\end{figure}
\end{example}
Let us then consider a one parameter family of sets $\pino_{t}$, $t\in [0,+\infty)$, obtained by ``elongating the nose'' of $\pino_{0}$ (see Figure~\ref{fig:pinocchio2}). More precisely, for $t\geq 0$ we define
\[
\pino_{t} = \pino(\theta_{0}) \cup \bigcup_{0<\tau<t} B_{r_0}(x_{\tau})\,,
\]
where $r_0=\sin\theta_0$ is the radius of the right half-ball in $\pino_0$, and $x_{\tau} = (\cos\theta_{0}+\tau,0)$. Since $\pino_t\supseteq \pino_0$ we have $h(\pino_t)\leq h(\pino_0)$, and we can immediately observe that
\begin{align}\label{eccoqua}
A(\pino_t)=A(\pino_0) + 2r_0 t\,, && P(\pino_t) = P(\pino_0) + 2t = \frac1{r_0} \,A(\pino_0) + 2t = \frac{A(\pino_t)}{r_0}\,.
\end{align}
Then, either $h(\pino_t)=h(\pino_0)$ and $\pino_t$ is Cheeger in itself, or $h(\pino_t)<h(\pino_0)$, and thus $\pino_t$ is not Cheeger in itself. We want to exclude this second possibility: indeed, if it were so, then a Cheeger set in $\pino_t$ should have some boundary in the interior of $\pino_t$, and this boundary should be made by arcs of circle with radius $1/h(\pino_t)>r_0$. But any such arc must necessarily start and end in the two corner points, because otherwise it should meet $\partial\pino_t$ tangentially and then it would be a complete circle, which is impossible because Lemma~\ref{lemma:180} ensures that any such arc must be at most half a circle. A Cheeger set in $\pino_t$ would then be contained in $\pino_0$, and this is impossible because $h(\pino_0)>h(\pino_t)$. The argument shows at once that $h(\pino_t)=h(\pino_0)$ for any $t>0$, and then that each $\pino_t$ is Cheeger in itself. However, from~(\ref{eccoqua}) we deduce that the Cheeger sets in $\pino_t$ are all the sets $\pino_\sigma$ for every $0\leq \sigma\leq t$. We have then found a set which admits a one-parameter family of Cheeger sets, as promised. Actually, owing to the properties of curved strips studied in Section~\ref{section:strip} (and in particular to Proposition~\ref{prop:stripAP}), the very same situation occurs even if one modifies the set $\pino_t$ by ``bending the nose''.

\begin{example}[A face with two stretched ears]\label{ex:ears}\rm 
We can push forward the construction of Example~\ref{ex:pinocchio} by adding a symmetric arc on the left of the unit disk, thus obtaining a ``face with two ears'' (see Figure~\ref{fig:2orecchie}). Let $\ears(\theta)$ denote this domain, where the geometric meaning of $\theta\in (0,\pi/2)$ is the same as in the previous example. By similar calculations as before, one finds
\[
P(\ears(\theta)) = 2(\pi - 2\theta) + 2\pi \sin\theta
\]
and
\[
A(\ears(\theta)= \pi - 2\theta + \sin(2\theta) + \pi \sin^{2}\theta\,.
\]
As before, one can check that there exists a unique solution $\theta_{1}$ of the equation
\[
\frac{P(\theta)}{A(\theta)} = \frac{1}{\sin\theta}
\]
in the interval $(0,\pi/2)$. Moreover, still arguing as in Example~\ref{ex:pinocchio}, one can prove that also $\ears(\theta_{1})$ is uniquely self-Cheeger. At this point, we can stretch independently the two ears producing a domain that contains a $2$-parameter family of Cheeger sets.\par
The interest of this example is not just having a two-parameter family of Cheeger sets instead of a one-parameter, but also that we have found a simply connected domain with Cheeger sets which have all non-empty intersection, but which are not all included into each other. In fact, as far as we know, in all the previous examples of domains with non-unique Cheeger sets, the different Cheeger sets either had empty intersection, as for instance in Example~\ref{ex:duecheeger} (taken from~\cite{KawLac2006}), or were all contained into each other, as for instance in Example~\ref{ex:pinocchio}.
\begin{figure}[ht]
\centering
\includegraphics[scale=1]{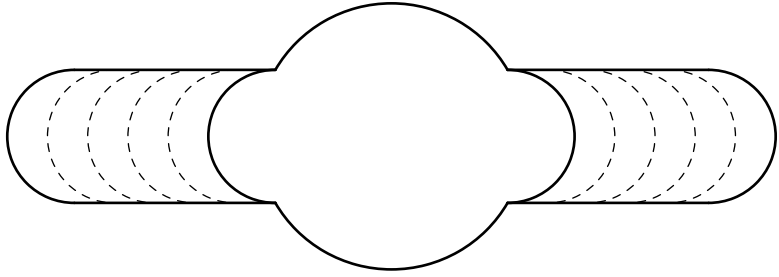}
\caption{A ``face with two stretched ears'', containing a two-parameter family of Cheeger sets. The minimal Cheeger set of the family corresponds to the white region (not foliated with half-circles) while the maximal one is the entire domain. A generic Cheeger set of the family is uniquely identified by independently choosing an arc in the foliation on the left and an arc in the foliation on the right.}
  \label{fig:2orecchie}
\end{figure}
\end{example}

\appendix

\section{Proof of property~(vii) in Proposition~\ref{prop:CheegerGenProp}}

We put here a proof of the property~(vii) of Proposition~\ref{prop:CheegerGenProp}, for the sake of completeness.

\begin{proof}
First we prove that, if we blow-up $E$ at $x\in \de^{*}\Om\cap \de E$, we obtain the same tangent half-space to $\Om$ at $x$. To this purpose, we need to show perimeter and volume density estimates for $E$ in $\Om$ at the point $x$. More precisely, we set $m(r) = |E\cap B_r(x)|$ and observe that, by our assumptions, $m(r)>0$ for all $r>0$. We now prove that 
\begin{equation}\label{densbelow1}
\liminf\limits_{r\to 0^{+}} \frac{m(r)}{r^{n}} >0\,.
\end{equation}
We choose a point $z\in \de^{*}E \cap \Om$ and, by standard estimates (see~\cite{MaggiBOOK}), we consider a one-parameter family of diffeomorphisms equal to the identity outside a small ball $B_\e(z)$, that allow us to produce any sufficiently small volume adjustment $\Delta V$ with a change in perimeter bounded by $c_{1}\Delta V$, for some constant $c_{1}>0$ depending only on $E$. Then for $r>0$ small enough we construct a competitor $F_{r}$ to $E$ such that $|F_{r}| = |E|$, $F_{r}\cap B_r(x) = \emptyset$, $F_{r} = E$ outside $B_r(x)\cup B_\e(z)$, and 
\[
P(F_{r}) \leq P(E\setminus B_r(x)) + c_{1}\, m(r)\,.
\]
On the other hand, we also have $P(E) \leq P(F_{r})$, thus for almost all $r>0$ we obtain
\begin{equation}\label{diffineq}
P(E;B_r(x)) \leq m'(r) + c_{1}\, m(r)\,.
\end{equation}
By the isoperimetric inequality~\eqref{isopRn} we find
\[
2m'(r) + c_{1}\, m(r) \geq n\om_{n}^{1/n} m(r)^{1-1/n}\,,
\]
thus for any $r>0$ small enough we obtain (hereafter $c$ denotes a small positive constant, possibly decreasing from line to line) 
\[
\frac{m'(r)}{m(r)^{1-1/n}} \geq c\,.
\]
By integrating between $\rho/2>0$ and $\rho$ we find 
\[
c\rho \leq (m(\rho)^{1/n} - m(\rho/2)^{1/n})\,,
\]
hence up to constants
\[
m(\rho) \geq c\rho^{n}
\]
for all $\rho>0$ small enough. This proves~\eqref{densbelow1}. By adjusting the volume of $E\setminus B_r(x)$ as before, we obtain a competitor $F_{r}$ that, owing to the minimality of $E$, allows us to show the existence of a positive constant $c_{2}>n\om_{n}$ such that for all $r>0$ small enough we have
\[
P(E;B_r(x)) \leq c_{2}r^{n-1}\,.
\]
Now we blow-up $\Om$ and $E$ at $x$. By our assumption on $x$, and thanks to Theorem~\ref{teo:degiorgi} (ii) and Proposition~\ref{prop:semicomp}, we get respectively a halfspace $x+H$ having $x$ on its boundary and, up to subsequences, a limit set $E_{\infty}$ contained in $x+H$ and with $x\in \de E_{\infty}$. One can show that $E_{\infty}$ is not empty and minimizes the perimeter without volume constraint and with respect to any compact variation $F$ contained in $x+H$. Since $H$ is convex, $E_{\infty}$ is also minimizing with respect to a generic compact variation in $\R^{n}$. By a maximum principle argument (see~\cite[Corollary~1]{Simon1987}) we get that $E_{\infty} = x+H$. This shows that $E$ admits the half-space $x+H$ as unique blow-up at $x$. We now prove that 
\begin{equation}\label{densitadisco}
\lim_{r\to 0}\frac{P(E;B_r(x))}{r^{n-1}} = \om_{n-1}\,.
\end{equation}
Indeed, we define $E_{r} = r^{-1}(E-x)$ and notice that $\chi_{E_{r}}\to \chi_{H}$ in $L^{1}_{loc}(\R^{n})$, as $r\to 0$. By Proposition~\ref{prop:semicomp} (i) we have 
\[
\liminf_{r\to 0}\frac{P(E;B_r(x))}{r^{n-1}} = \liminf_{r\to 0} P(E_{r};B_1(0)) \geq P(H;B_1(0)) = \om_{n-1}\,,
\]
thus to prove~\eqref{densitadisco} we only need to show that
\begin{equation}\label{upperdensitadisco}
\limsup_{r\to 0}P(E_{r};B_1(0))\leq \om_{n-1}\,.
\end{equation}
Let us assume by contradiction that there exists $\e>0$ and a sequence of radii $r_{i}\to 0$ as $i\to\infty$, such that setting $E_{i} = E_{r_{i}}$ and $\Om_{i} = r_{i}^{-1}(\Om-x)$ we have for all $i\in \N$
\begin{equation}\label{assurdodensitadisco}
P(E_{i};B_1(0))\geq \om_{n-1}+\e\,.
\end{equation}
Notice that, for $i$ large enough, one has
\begin{equation}\label{omegaokkei}
P(\Om_{i};B_s(0)) \leq s^{n-1}(\om_{n-1} + \e/3)\qquad \text{for all } 1<s<2\,.
\end{equation}
Since $\chi_{E_{i}}\to \chi_{H}$ in $L^{1}(B_2(0))$ as $i\to \infty$, by the Coarea formula we find some
\[
t\in\bigg(1,\left(\frac{\om_{n-1}+\e/2}{\om_{n-1}+\e/3}\right)^{\frac{1}{n-1}}\bigg)
\]
such that we have
\begin{gather}
\label{urka1}
P(\Omega_{i};\de B_t(0))=P(E_{i};\de B_t(0)) = 0\,,\\
\label{urka2}
\big|E_{i}\difsim \Om_{i}\cap B_2(0)\big| < \e^3 \,,\\ 
\label{urka3}
\Hau^{n-1}\Big(E_{i}\difsim \Om_{i} \cap \de B_t(0)\Big) < \frac{\e}{4}\,.
\end{gather}
Consider now the set
\[
\widehat F_{i} = \big(E\setminus B_{tr_{i}}(x)\big) \cup \big(\Om\cap B_{tr_{i}}(x)\big)\,,
\]
and notice that by~(\ref{urka1}) and~(\ref{urka3}) one has
\begin{equation}\label{perFi}
P(\widehat F_{i};\Om) = P\big(E;\Om\setminus B_{tr_{i}}(x)\big) + P(\Om;B_{tr_{i}}(x))
+r_i^{n-1}\Hau^{n-1}\Big(E_{i}\difsim \Om_{i} \cap \de B_t(0)\Big)\,,
\end{equation}
while by~(\ref{urka2})
\[
\Big|\big|\widehat F_i\big| - \big|\widehat E\big|\Big| \leq \eps^3 r_i^n\,.
\]
If now $F_i$ is a set with the same volume as $E$ obtained, as before, with a small adjustement of $\widehat F_i$ far from $B_{2r_i}(x)$, we get
\begin{equation}\label{15}
P(F_i)\leq P(\widehat F_i) + c_1 \eps^3 r_i^n\,.
\end{equation}
Finally, combining~\eqref{assurdodensitadisco}, \eqref{omegaokkei}, \eqref{urka1}, \eqref{urka3}, \eqref{perFi} and~\eqref{15} with the minimality of $E$ and the bound on $t$, we get for $r_{i}$ small enough that
\[\begin{split}
r_{i}^{n-1}(\om_{n-1}+\e) &\leq P(E;B_{r_{i}}(x))\leq P(E;B_{tr_{i}}(x))
\leq (tr_{i})^{n-1}(\om_{n-1} + \e/3) + \frac{\e}{4} r_{i}^{n-1} + c_{1} \e^3 r_{i}^{n}\\
&\leq (tr_{i})^{n-1}(\om_{n-1} + \e/3) + \frac{\e}{2} r_{i}^{n-1}< r_{i}^{n-1}(\om_{n-1} + \e)\,,
\end{split}\]
which is a contradiction. This proves~\eqref{upperdensitadisco}, thus~\eqref{densitadisco}. To show that $x\in \de^{*}E$ and $\nu_{E}(x) = \nu_{\Om}(x)$ we set $v = -\nu_{\Om}(x) = -\nu_{H}(0)$ and, owing to~\eqref{densitadisco}, we just have to show that
\begin{equation}\label{ridotta1}
\lim_{r\to 0} \frac{D\chi_{E}(B_r(x))\cdot v}{\om_{n-1}r^{n-1}} = 1\,.
\end{equation}
In fact, by Theorem~\ref{teo:degiorgi} (iv) we have for almost all $r>0$ that
\begin{equation}\begin{split}\label{dechiEv}
D\chi_{E}(B_r(x))\cdot v &= \int_{E\cap \de B_r(x)} v\cdot N\, d\Hau^{n-1}
= \int_{H\cap \de B_r(0)} v\cdot N\, d\Hau^{n-1} + A(x,r)\\
&= \om_{n-1}r^{n-1} + A(x,r)\,, 
\end{split}\end{equation}
where $N$ is the exterior normal to $\de B_r(x)$ and 
\[\begin{split}
|A(x,r)| &= \left|v\cdot \int_{\de B_r(x)} (\chi_{E}(y) - \chi_{x+H}(y)) N(y)\, d\Hau^{n-1}(y)\right|\\ 
&\leq \int_{\de B_r(x)} |\chi_{E}(y) - \chi_{x+H}(y)| \, d\Hau^{n-1}(y)\,.
\end{split}\]
For any fixed $\delta>0$, we define the set $\Sigma(x,\delta)\subseteq(0,+\infty)$ of radii $r>0$ such that $A(x,r)>\delta r^{n-1}$. Then by the $L^{1}_{loc}$-convergence of $r^{-1}(E-x)$ to the half-space $H$ we infer that
\[
\lim_{\rho\to 0^{+}}\frac{\Hau^{1}(\Sigma(x,\delta)\cap (0,\rho))}{\rho} = 0\,.
\]
Consequently, for any decreasing infinitesimal sequence $(r_{i})_{i}$ we can find another sequence $(\rho_{i})_{i}$ such that $\rho_{i}\notin \Sigma(x,\delta)$ for all $i$ and $\rho_{i} = r_{i}+o(r_{i})$ as $i\to\infty$. Let us now assume by contradiction that~\eqref{ridotta1} does not hold. Then there must exist $\alpha>0$ and a decreasing infinitesimal sequence $(r_{i})_{i}$, such that
\begin{equation}\label{DchiEassurdo}
\left|\frac{D\chi_{E}(B_{r_{i}}(x))\cdot v}{\om_{n-1}r_{i}^{n-1}} - 1\right|\geq \alpha
\end{equation}
for all $i\in \N$. We now choose $\delta = \frac{\alpha}2\om_{n-1}$ and consider the infinitesimal sequence $\rho_{i}$ as above. By~\eqref{dechiEv} with $\rho_{i}$ replacing $r$, we have that
\[
\Big|D\chi_{E}(B_{\rho_{i}}(x))\cdot v  -\om_{n-1}\rho_{i}^{n-1}\Big| = |A(x,\rho_{i})| \leq \frac{\alpha}{2}\om_{n-1} \rho_{i}^{n-1}\,.
\]
On the other hand, by~\eqref{densitadisco} we also have
\[\begin{split}
\Big|D\chi_{E}(B_{\rho_i}(x)) - D\chi_{E}(B_{r_{i}}(x))\Big| &\leq P(E;B_{\rho_{i}}(x)\difsim B_{r_{i}}(x))
\leq \om_{n-1}|\rho_{i}^{n-1} - r_{i}^{n-1}| +  o(r_{i}^{n-1})\\
&= o(r_{i}^{n-1})
\end{split}
\]
as $i\to \infty$. By combining the last two inequalities we get
\[
\Big|D\chi_{E}(B_{r_{i}}(x))\cdot v - \om_{n-1}r_{i}^{n-1}\Big|\leq \frac{\alpha}{2}\om_{n-1} \rho_{i}^{n-1} + o(r_{i}^{n-1}) = \frac{\alpha}{2}\om_{n-1} r_{i}^{n-1} + o(r_{i}^{n-1})\,,
\]
which contradicts~\eqref{DchiEassurdo} for $i$ large enough. This concludes the proof of (vii).
\end{proof}

%\section{Da fare?}
%\begin{itemize}
%\item Bootstrap per migliorare l'espansione della costante dello strip (dato che il bound $O(L^{-2})$ in parte usa quello $O(L^{-1})$, forse si puo' iterare).
%\item Costruzione di un dominio con famiglia di Cheeger a due parametri.
%\item Se l'inner Cheeger \`e connesso, allora ha reach pari a $r$.
%\item Stabilit\`a del pinocchio.
%\item Esistenza di domini con un infinito numerabile di Cheeger sets.
%\item Confronto coi risultati di Concus, Finn \& altri sulla capillarita' (vedere in particolare~\cite{ConFinWei2000,Finn1984,CheConFin1995,Chen1980}).
%\end{itemize}

\section*{Acknowledgements}
This work has been partially supported by ERC Starting Grant 2010 ``AnOptSetCon''. We also thank Carlo Nitsch and Antoine Henrot for their useful comments.

\end{document}